\documentclass[12pt]{article}
\usepackage{graphicx}
\usepackage{amsmath}
\usepackage{amsthm}
\usepackage{graphics}
\usepackage{epsfig}
\usepackage{amssymb}
\usepackage{enumitem}
\usepackage{hyperref}
\usepackage{dsfont}
\usepackage{float}
\usepackage{verbatim}
\usepackage{authblk}
\usepackage{nicematrix}
\usepackage{booktabs}

\usepackage[toc,page]{appendix}
\usepackage{xcolor}
 \usepackage{setspace}
\usepackage{extarrows}

\usepackage[sort&compress]{natbib}
\usepackage{subfig}
\usepackage{lmodern}

\usepackage{pdflscape}
\usepackage{enumitem}
\allowdisplaybreaks

\newcommand{\lcb}{\left\{}
\newcommand{\rcb}{\right\}}

\let\originalleft\left
\let\originalright\right
\renewcommand{\left}{\mathopen{}\mathclose\bgroup\originalleft}
\renewcommand{\right}{\aftergroup\egroup\originalright}

\newcommand{\Var}{\operatorname{Var}}

\newcommand\as{\operatorname{a. s.}}
\newcommand{\tr}{\operatorname{tr}}
\newcommand{\supp}{\operatorname{supp}}

\newcommand{\cond}{\stackrel{\mathcal{D}}{\to}}
\newcommand{\conp}{\stackrel{\mathbb{P}}{\to}}
\newcommand{\conas}{\stackrel{\as}{\to}}

\newcommand{\diag}{\operatorname{diag}}

\newcommand{\bfA}{\mathbf{A}}

\newcommand{\bfI}{\mathbf{I}}
\newcommand{\bfSigmahat}{	\hat{ \mathbf{\Sigma} }}

\newcommand{\bfSigma}{\mathbf{\Sigma}}

\newtheorem{theorem}{Theorem}[section]
\newtheorem{lemma}{Lemma}[section]

\newtheorem{corollary}{Corollary}[section]

\newtheorem{remark}{Remark}[section]

\numberwithin{equation}{section}

\allowdisplaybreaks

\begin{document}
\title{Central limit theorems for high-dimensional ridge regression under uniform subsampling}
\date{\today}
\author[1]{Thomas Lam}
\author[2]{Nina Dörnemann}
\author[1]{Holger Dette}

\affil[1]{Department of Mathematics, Ruhr University Bochum, Germany}
\affil[2]{Department of Mathematics, Aarhus University, Denmark}

\maketitle
 \begin{abstract}
 {\noindent We study ridge regression based on uniformly subsampled observations in high-dimensional linear models, where the number of predictors and the sample size grow proportionally. Subsampling is implemented through independent Bernoulli indicators with a common selection probability. We establish central limit theorems for linear contrasts of the ridge estimator, allowing for non-Gaussian predictors and general population covariance matrices. Under a vanishing-signal assumption, we obtain an explicit bias correction in the isotropic case and develop consistent estimators of the asymptotic variances.  These results enable asymptotically valid tests for linear contrasts of the regression parameter in the isotropic setting. Our proofs exploit the resolvent representation of the ridge estimator and combine leave-one-out arguments with deterministic equivalents from random matrix theory.}
 
%    We study ridge regression for uniformly subsampled observations in high-dimensional linear models, where the number of predictors and the sample size grow proportionally. The subsampling mechanism is modeled by independent and identically distributed Bernoulli random variables. 
%    We establish a central limit theorem for linear contrasts of the ridge estimator in a vanishing-signal regime, where the Euclidean norm of the regression parameter converges  to zero.
%  Our analysis neither requires  Gaussian predictors nor an isotropic population covariance matrix. We construct a consistent estimator of the asymptotic variance of the ridge estimator. In the isotropic case, we further derive an asymptotically valid hypothesis test for linear contrasts of the regression parameter. Our analysis builds on the resolvent representation of the ridge estimator. We combine leave-one-out arguments with deterministic equivalents from random matrix theory to establish asymptotic normality and characterize the asymptotic variance.
  \end{abstract}

MSC2020 subject classifications: Primary 62J07, 62F12; secondary 60B20.


\vspace{1cm}

\section{Introduction}

Statistical inference in high-dimensional linear models is challenging because regularization introduces non-negligible bias in the estimation of linear contrasts of the regression parameter. A substantial literature has investigated the asymptotic properties of ridge regression in the proportional high-dimensional regime, including prediction risk and distributional results; see, for example, \cite{Dicker2016,DobribanWager2018,Thrampoulidis2018,HastieEtAl2022,BellecZhang2023,han:xu:2026}. While recent contributions allow for non-Gaussian predictors and general population covariance matrices,  distributional results often rely on strong tail assumptions, such as the sub-Gaussian conditions in \cite{han:xu:2026}. 

In this paper, we investigate the asymptotic distribution of linear contrasts of ridge estimators under uniform subsampling, allowing for general population covariance matrices and requiring only bounded eighth moments of the standardized predictor coordinates. Subsampling reduces computational costs by fitting models to randomly selected subsets of the observations;
%Statistical inference in high-dimensional linear models requires methods that account for the effects of regularization on estimationuncertainty. Ridge regression is widely used in this setting,but the bias introduced by regularization poses challenges for inference on linear contrasts of the regression parameter. Several existing asymptotic results for regularized estimators rely on Gaussian predictors or isotropic population covariance matrices; see, for example, \cite{Dicker2016, Thrampoulidis2018, DobribanWager2018,BellecZhang2023,han:xu:2026}.
%This motivates developing distributional guarantees under more general assumptions on the predictor distribution and covariance structure. 
%We combine this question with subsampling, a widely used approach to reducing computational costs by fitting models to a randomly selected subset of the observations; 
see, for example, \cite{MaMahoneyYu2015,CohenEtAl2015,LeeNg2020}. For ridge regression, subsampling changes the effective dimension-to-sample-size ratio and consequently affects both the regularization bias and the variance of the estimator. Under Bernoulli sampling, the subsample size is itself random. A central question is therefore how to account for these effects in the asymptotic distribution of linear contrasts under general assumptions on the predictor distribution and covariance structure.
% We combine this question with another statistical approach, namely random  subsampling, which offers a practical way to reduce computational costs by constructing estimators from a smaller number of observations. \ND{reference for subsampling}  For ridge regression, this changes the effective dimension-to-sample-size ratio and consequently the regularization bias and variance of the estimator. Under Bernoulli sampling, the subsample size is itself random. A central issue is therefore how to account for these effects in the centering and variance of linear contrasts when the predictors are neither Gaussian nor isotropic.
Our main contributions can be summarized as follows:
\begin{itemize}
   \item 
   We establish central limit theorems for appropriately centered linear contrasts of the subsampled ridge estimator, allowing for non-Gaussian predictors and general population covariance matrices
(Theorems~\ref{thm_main_CLT_identity_cov} and~\ref{thm_CLT_arbitrary_cov}). In the isotropic case, a vanishing-signal assumption allows us to obtain an explicit bias correction.
 %  We establish central limit theorems for appropriately centered linear contrasts of the subsampled ridge estimator in a vanishing-signal regime, where the Euclidean norm of the regression parameter tends to zero. Our results allow for non-Gaussian predictors and both spherical and general covariance matrices (Theorems~\ref{thm_main_CLT_identity_cov} and~\ref{thm_CLT_arbitrary_cov}).

\item We construct consistent variance estimators in both covariance settings (Theorem~\ref{thm_consistent_variance_estimator_identity_cov} and Corollary~\ref{cor_variance_estimator_arbitrary_cov}). These estimators combine an estimator of the regression error variance with an observable quadratic form involving the subsampled sample covariance matrix.

\item In the isotropic  covariance setting, we develop a hypothesis test for linear contrasts of the regression parameter with asymptotically correct size (Theorem~\ref{thm_test_normal} and Corollary~\ref{cor_test_performance}). We establish consistency against sufficiently separated alternatives and characterize its asymptotic power under local alternatives (Corollary~\ref{cor_test_performance}).
\end{itemize}

Our theoretical analysis builds on the resolvent representation
of the ridge estimator. We combine leave-one-out arguments with recent results on
deterministic equivalents in
\cite{StarkDoernemann2026} to establish asymptotic normality
and characterize the relevant variance quantities. The following subsection introduces the model and estimator
and provides an overview of the main results.

\subsection{Overview of the main results} \label{sec_overview}
While the precise statements of our results require additional
notation and assumptions, their main features can be presented
using only minimal notation.
Consider the linear model
\[
    Y_{i,n} = X_{i,n}^\top\beta_0+\varepsilon_{i,n},
    \qquad i=1,\ldots,n,
\]
where $X_{i,n}\in\mathbb{R}^p$ has mean zero and population covariance
matrix $\bfSigma_n$. We work in the proportional regime
$p/n\to\kappa\in(0,\infty)$. Observations are selected independently
with probability $r_n\to R\in(0,1]$, using  independent Bernoulli indicators
$\delta_{1,n},\ldots,\delta_{n,n}$ that are independent of the data.
For a fixed regularization parameter $\tau>0$, the subsampled ridge
estimator is
\begin{align} \label{eq_min_prob}
    \widehat\beta
    =
    \operatorname*{arg\,min}_{b\in\mathbb{R}^p}
    \left\{
        \frac{1}{n}\sum_{i=1}^n
        \delta_{i,n}(Y_{i,n}-X_{i,n}^\top b)^2
        +\tau\|b\|^2
    \right\}.
\end{align}
Our interest lies in linear contrasts $c_n^\top\widehat\beta$
for prescribed deterministic vectors $c_n\in\mathbb{R}^p$.
Under suitable assumptions, we establish central limit theorems
of the schematic form
\begin{align} \label{eq:clt_overview}
    \frac{\sqrt{nr_n}}{s_n\|c_n\|}
    c_n^\top(\widehat\beta-\bfA_n\beta_0)
    \cond \mathcal{N}(0,1), \quad n,p\to\infty.
\end{align}
where the observable $p\times p$ matrix $\bfA_n$ accounts for the bias induced by regularization
and $s_n^2$ denotes the corresponding variance. We construct a consistent estimator
$\widehat s_n^2$ of $s_n^2$; see
Theorem~\ref{thm_consistent_variance_estimator_identity_cov}
and Corollary~\ref{cor_variance_estimator_arbitrary_cov}. 
The precise form of $\bfA_n$ depends on the covariance structure.
When $\bfSigma_n=\bfI_p$, the centering term $c_n^\top \bfA_n\beta_0$ is a scalar multiple of
$c_n^\top\beta_0$ (Theorem~\ref{thm_main_CLT_identity_cov}).
Consequently, under the null hypothesis
\[
    H_{0,n}:\ c_n^\top\beta_0=\theta_n,
\]
the CLT in \eqref{eq:clt_overview}, after replacing $s_n$ by $\hat s_n$, yields an asymptotically valid test (Theorem~\ref{thm_test_normal}
and Corollary~\ref{cor_test_performance}). 
For general population covariance matrices $\bfSigma_n$, the
matrix $\bfA_n$ takes a more complex form involving the resolvent of the subsample covariance matrix (Theorem~\ref{thm_CLT_arbitrary_cov}).

\subsection{Related work}\label{sec_related_work}

We review related work on high-dimensional regression and  statistical
inference, and conclude with a more detailed comparison
between our contributions and the most closely related results.

The asymptotic behavior of ridge-regularized estimators in the zero-signal setting has been investigated by \cite{ElKaroui2013}. Subsequently, \cite{ElKaroui2018} extended this analysis to nonzero regression parameters for elliptical-like predictors. Under suitable assumptions, the asymptotic estimation error was characterized and a central limit theorem for individual coordinates of the estimator was established. The class of loss functions considered in the latter work has bounded derivatives and does not include the squared loss. Furthermore, \cite{Thrampoulidis2018} characterized the asymptotic squared estimation error for a broader class of convex regularization, including ridge, under Gaussian designs. Related risk results for ridge regression with weighted $\ell_2$-regularization were obtained by \cite{WuXu2020}. More recently, \cite{han:xu:2026} established distributional
approximations for ridge and ridgeless estimators under sub-Gaussian
designs with general population covariance matrices. Their results
include weighted $\ell_q$ risk characterizations and confidence
intervals with asymptotically correct average coverage across
coordinates.
% For the Lasso, introduced by \cite{Tibshirani1996}, several debiasing methods have been proposed. For instance, \cite{ZhangZhang2014} proposed a low-dimensional projection approach, while \cite{javanmard:montanari:2014} developed a method based on a convex optimization problem. Both approaches enable asymptotically valid statistical inference under suitable assumptions.
A further line of research concerns
statistical inference for regularized estimators.
\cite{BellecZhang2023} developed debiasing procedures for convex regularized estimators, including ridge regression, and established asymptotic normality for most normalized contrast directions under suitable Gaussian assumptions. In single-index models with Gaussian covariates, \cite{Bellec2025} established central limit theorems for most coordinates of suitably adjusted regularized $M$-estimators. For fixed design matrices, \cite{GaoTsay2024} proposed a debiasing method based on ridge regression and ridge screening that enables inference for linear functionals under sparsity conditions.

Estimation of the regression error variance provides another
related line of research. Under a vanishing-signal assumption, \cite{LiuZhengFeng2020} proposed a consistent estimator of the regression error variance based on ridge regression in the high-dimensional regime. \cite{ChoiPark2024} further studied this estimator beyond the vanishing-signal setting. 

%Subsampling has been investigated in both classical and high-dimensional settings. In the classical fixed-dimensional setting, \cite{ChenDetteYu2024} proposed a subsampling approach for linear classification problems and established asymptotic properties and efficiency improvements for the resulting estimators. 
In the high-dimensional framework, subsampling has been investigated in connection with bagging procedures. In particular, \cite{Patil2023} characterized the asymptotic prediction risk of bagged ridge and ridgeless estimators, while \cite{Bellec2026} studied the dependence between estimators constructed from overlapping subsamples in robust linear and logistic regression. %A related approach to sampling for the estimation of linear contrasts is the $c$-optimal design framework of \cite{EftekhariBanerjeeRitov2023} for Lasso estimation. 

\vspace{0.5cm}

Among these contributions, we discuss three particularly closely
related works in more detail to clarify their connection to our results.
\cite{BellecZhang2023} and \cite{Bellec2025} developed  central limit theorems and inferential
procedures for broad classes of regularized estimators under Gaussian designs. In contrast, our analysis focuses on ridge regression in linear models
and allows for non-Gaussian predictors with general population covariance matrices.  The connection to \cite{Bellec2026} concerns subsampling in the proportional regime. That work studies the dependence between estimators fitted to overlapping subsamples of fixed size and its implications for
bagging. Our focus is instead on central limit theorems and consistent
variance estimation for linear contrasts of a single ridge estimator
under Bernoulli subsampling.
To the best of our knowledge, our results provide the first central limit theorems and consistent variance estimators for linear contrasts of ridge estimators under Bernoulli subsampling, allowing for non-Gaussian predictors with general population covariance matrices under only bounded eighth-moment assumptions.
\section{Theoretical Results}
% \textbf{Notations}
% Let $X_i\in \mathbb{R}^{p}$ be a vector of predictors and $Y_i\in \mathbb{R}$ a response. We assume that the predictors and responses have a linear relation. To be more precise, 
% \begin{align*}
%     Y_i=\varepsilon_i + X_i^\top \beta_0    ~~~~(i=1, \ldots , n) ,
% \end{align*}
% where $\varepsilon_i$ are unknown random errors and $\beta_0\in\mathbb{R}^{p}$ is the true regression vector. Let $\hat{\beta}$ be the ridge estimator based on uniformly subsampled data. To be more precise,
% \begin{align}\label{eq_min_prob}
%     \hat{\beta}=\argmin\limits_{\beta \in \mathbb{R}^p} \frac{1}{n}\sum_{i=1}^{n} \delta_i \left(Y_i - X_i^\top \beta\right)^2 + \tau \left\Vert \beta\right\Vert^2,
% \end{align}
% where $\tau >0$ is a constant and $\delta_1, \ldots, \delta_n$ are i.i.d. Bernoulli distributed random variables with success probability $\mathbb{P} (\delta_i = 1) =r$ and $\| \cdot\|$ denotes the Euclidean norm. 
We consider the model and subsampled ridge estimator introduced
in Section~\ref{sec_overview}. We now state the assumptions and
introduce the additional notation needed for our main results.
For square matrices $\mathbf{A}$, $\Vert \mathbf{A} \Vert_2$ denotes the spectral norm and $\Vert \mathbf{A} \Vert_{F}$ the Frobenius norm.
\medskip

\textbf{Assumptions}
\begin{enumerate}[label=(A\arabic*)]
    \item \label{ass_sample_to_dimension_ratio}As $n \to \infty$, we have $p=p(n)$ with $\kappa_n:= p/n \to \kappa \in (0, \infty)$ and let $\delta_{1,n},\ldots,\delta_{n,n}$ 
    be independent identically Bernoulli distributed random variables with success probability $r_n$, where  
    $r_n \to R \in (0,1]$. 
    \item \label{ass_random_vectors_MP} We assume that the random vectors $X_{i,n}$ can be represented as
    \begin{align*}
        X_{i,n} = \bfSigma_n^{1/2} Z_{i,n},
    \end{align*}
    where $\bfSigma_n^{1/2}$ is the nonnegative definite square root of the covariance matrix $\bfSigma_n \in \mathbb{R}^{p\times p}$ and $Z_{i,n}$ is a $p$-dimensional random vector. The random vectors $Z_{1,n}, \ldots ,Z_{n,n}$ are assumed to be independent and identically distributed. Furthermore, for each $n\in \mathbb{N},$ we assume the entries of $Z_{i,n}$ are real-valued independently distributed random variables satisfying $\mathbb{E}[Z_{i,n}(j)]=0$, $\mathbb{E}[Z_{i,n}(j)^2]=1$ and
        \begin{align*}
           M_8:=\sup_{n\in \mathbb{N}}\max_{1 \leq i \leq n, 1 \leq j \leq p} \mathbb{E}[Z_{i,n}(j)^8]&<\infty.
        \end{align*}
    \item \label{ass_spectral_norm_cov_matrix}The matrices $\bfSigma_n$ satisfy
        \begin{align*}
            \sup_{n\in \mathbb{N}} \left\Vert \bfSigma_n\right\Vert_2<\infty.
        \end{align*}
        \item\label{ass_ESD_conv} The sequence of empirical spectral distributions $H_n = F_{\bfSigma_n}$ converges weakly to a probability distribution $H$, i.e
        \begin{align*}
            H_n \cond H.
        \end{align*}
    \item \label{ass_random_errors} For each $n\in \mathbb{N}$, let $\varepsilon_{1,n}, \ldots, \varepsilon_{n,n}$ be independent and identically distributed random errors with $\mathbb{E}[\varepsilon_{1,n}] = 0, \Var(\varepsilon_{1,n})= \sigma_{\varepsilon}^2 \in (0, \infty)$ and $\sup_{n\in \mathbb{N}}\mathbb{E}[\varepsilon_{1,n}^4]<\infty$.
    %For $i=1, \ldots, n,$ the random errors $\varepsilon_i$ are independent and identically distributed with $\mathbb{E}\left[\varepsilon_i\right]=0, \Var\left(\varepsilon_i\right)= \sigma_{\varepsilon}^2 >0 $ and $\mathbb{E}\left[\varepsilon_i^4\right]<\infty$.
    \item \label{ass_mutual_independence}
    The random variables $\lcb X_{i,n}\rcb_{i=1}^{n}$, $\lcb \varepsilon_{i,n}\rcb_{i=1}^{n}$ and $\lcb \delta_{i,n}\rcb_{i=1}^{n}$ are mutually independent.
\end{enumerate}
\medskip

We now start with the analysis of the the ridge estimator considering the isotropic and general case separately.  Observing  the minimization problem \eqref{eq_min_prob}, straightforward algebra yields
\begin{align}
    \hat{\beta} &= \left(\frac{1}{n}\mathbf{X}_n^\top {\Delta}_n \mathbf{X}_n + \tau \mathbf{I}_p\right)^{-1}\frac{1}{n}\mathbf{X}_n^\top {\Delta}_n \mathcal{Y}_n = \left(\frac{1}{n}\mathbf{X}_n^\top {\Delta}_n\mathbf{X}_n+ \tau \mathbf{I}_p\right)^{-1}\frac{1}{n}\mathbf{X}_n^\top {\Delta}_n \left(\mathbf{X}_n \beta_0 + \mathcal{E}_n\right)\nonumber
    \\
    &= \beta_0 - \tau \left(\frac{1}{n}\mathbf{X}_n^\top {\Delta}_n \mathbf{X}_n + \tau \mathbf{I}_p\right)^{-1}\beta_0 + \left(\frac{1}{n}\mathbf{X}_n^\top {\Delta}_n \mathbf{X}_n + \tau \mathbf{I}_p\right)^{-1}\frac{1}{n}\mathbf{X}_n^\top \Delta_n \mathcal{E}_n,\nonumber
    \\
    &= \beta_0 - \tau \mathbf{R}(\mathbf{B}_n,\tau)\beta_0 + \mathbf{R}(\mathbf{B}_n,\tau)\frac{1}{n}\mathbf{X}_n^\top \Delta_n \mathcal{E}_n,\label{eq_rep_ridge_estimator}
\end{align}
where
\begin{align*}
    \mathbf{X}_n^\top &= \left(X_{1,n},\ldots, X_{n,n}\right) \in \mathbb{R}^{p \times n},
    \\
    \mathbf{I}_p &= \diag(1,\ldots,1) \in \mathbb{R}^{p\times p},
    \\
    \Delta_n &= \diag\left(\delta_{1,n}, \ldots, \delta_{n,n}\right) \in \mathbb{R}^{n\times n},
    \\
    \mathcal{Y}_n &= (Y_{1,n},\ldots,Y_{n,n})^\top\in \mathbb{R}^n
    \\
    \mathcal{E}_n &= \left(\varepsilon_{1,n}, \ldots, \varepsilon_{n,n}\right)^\top \in \mathbb{R}^n,
    \\
    \mathbf{B}_n &= \frac{1}{n}\mathbf{X}_n^\top \Delta_n \mathbf{X}_n \in \mathbb{R}^{p\times p},
    \\
    \mathbf{R}(\mathbf{A},z) &= \left(\mathbf{A}+z \mathbf{I}_p\right)^{-1} \in \mathbb{R}^{p\times p}.
\end{align*}

\subsection{The isotropic case: \texorpdfstring{$\bfSigma_n = \mathbf{I}_p$}{Identity Matrix}}\label{sec_identity}
In this section, we assume that $\bfSigma_n=\mathbf{I}_p$. Hence $H_n = H = \delta_{\{1\}}$, which denotes the Dirac measure with mass in point $1$.
Let \begin{align}\label{eq_def_v_function_identity_cov}
    v(\kappa, z) = \frac{-(1-\kappa +z)-\sqrt{(1-\kappa+z)^2-4z}}{2z}, \quad z\in \mathbb{C}\setminus\mathbb{R}_{\geq 0},
\end{align}
denote the Stieltjes transform of the companion Marchenko-Pastur distribution $\underline{F}^{\kappa, \delta_{\{1\}}}$,  where the complex root is chosen such that $\Im(v(\kappa,z))\Im(z)>0$ for $z\in \mathbb{C}\setminus \mathbb{R}$, and denote its derivative by 
    \begin{align}\label{eq_def_v_derivative_function_identity_cov}
        v^\prime(\kappa,z)=\frac{d}{dz}v(\kappa,z)=\frac{1}{2}\left\{\frac{-(z-\kappa-1)}{z\sqrt{(1-\kappa+z)^2-4z}}+\frac{1-\kappa+\sqrt{(1-\kappa+z)^2-4z}}{z^2}\right\}, ~~ z\in \mathbb{C}\setminus \mathbb{R}_{\geq 0}.
    \end{align}
Furthermore, let $v_n(\kappa_n,z)$ denote the Stieltjes transform of the companion Marchenko-Pastur distribution $\underline{F}^{\kappa_n,H_n}$. Since $H_n=H=\delta_{\{1\}}$, we have $v_n(\kappa_n,z) = v(\kappa_n,z)$.
    \begin{theorem}\label{thm_main_CLT_identity_cov}
    Suppose assumptions \ref{ass_sample_to_dimension_ratio} - \ref{ass_mutual_independence} are satisfied. Let $c_n\in \mathbb{R}^p\setminus \{\mathbf{0}_p\}$ and suppose that $\Vert \beta_0 \Vert = o(1)$. Then,
    \begin{align*}
        \frac{\sqrt{r_n n}}{\Vert c_n\Vert} c_n^\top \left(\hat{\beta}-\left(1-a_n(N_n)\right)\beta_0\right)\cond \mathcal{N}(0, \sigma^2),
    \end{align*}
    with $\sigma^2$ given by
    \begin{align*}
        \sigma^2=\sigma_{\varepsilon}^2\frac{v^\prime\left(\frac{\kappa}{R},-\frac{\tau}{R}\right)}{\left(v\left(\frac{\kappa}{R},-\frac{\tau}{R}\right)+1\right)^2} > 0
    \end{align*}
    and
    \begin{align*}
        a_n(N_n) = \begin{cases}
            1 & N_n = 0
            \\
            \frac{1}{v\left(\frac{p}{N_n}, -\frac{n}{N_n}\tau\right)+1} & N_n >0
        \end{cases},
    \end{align*}
    where $N_n=\sum\limits_{i=1}^{n} \delta_{i,n}$.
\end{theorem}

For the purpose of estimating $\sigma^2$ consistently, we introduce the matrix 
\begin{align*}
    \mathbf{M}_n = \frac{1}{n}\Delta_n\mathbf{X}_n\mathbf{R}(\mathbf{B}_n,\tau)\mathbf{X}_n^\top\Delta_n.
\end{align*}
\begin{theorem}\label{thm_consistent_variance_estimator_identity_cov}
    Suppose assumptions \ref{ass_sample_to_dimension_ratio} - \ref{ass_mutual_independence} are satisfied. Let $c_n\in \mathbb{R}^{p}\setminus\{\mathbf{0}_p\}$ and suppose $\Vert \beta_0\Vert = o(1)$. Then
    \begin{align*}
        \hat{\sigma}_n^2 = r_n\hat{\sigma}_{n,1}^{2} \hat{\sigma}_{n,2}^{2} \conp \sigma^2,
    \end{align*}
    where 
    \begin{align*}
        \hat{\sigma}_{n,1}^{2} = \begin{cases}
            0 & N_n = 0
            \\
            \frac{1}{n^{-1}(N_n-\tr\left(\mathbf{M}_n\right))}\frac{1}{n}\mathcal{Y}_n^\top  \left(\Delta_n- \mathbf{M}_n\right) \mathcal{Y}_n & N_n >0
        \end{cases}   = \begin{cases}
            0 & N_n = 0
            \\
            \frac{1}{N_n -\tr(\mathbf{M}_n)}\mathcal{Y}_n^\top (\Delta_n -\mathbf{M}_n)\mathcal{Y}_n & N_n >0
        \end{cases}
    \end{align*}
    and
    \begin{align*}
        \hat{\sigma}_{n,2}^{2}
        &=\frac{1}{\Vert c_n\Vert^2}c_n^\top \left[\mathbf{R}(\mathbf{B}_n, \tau)-\tau\mathbf{R}(\mathbf{B}_n,\tau)^2\right]c_n
    \end{align*}
    and $\sigma^2$ is defined in Theorem \ref{thm_main_CLT_identity_cov}.
\end{theorem}

\begin{remark}
    Since
    \begin{align*}
        \mathbf{M}_n = \frac{1}{n}\Delta_n \mathbf{X}_n\mathbf{X}_n^\top \Delta_n \left(\frac{1}{n}\Delta_n \mathbf{X}_n \mathbf{X}_n^\top \Delta_n + \tau \mathbf{I}_n\right)^{-1},
    \end{align*}
    it follows on $\{N_n >0\}$ for every $n$ that
    \begin{align}\label{eq_TL1}
        \frac{N_n - \tr(\mathbf{M}_n)}{n} = \frac{1}{n}\sum_{i=1}^{N_n} \frac{\tau}{\lambda_i +\tau} >0,
    \end{align}
    where $\lambda_1, \ldots, \lambda_{N_n}$ denote the eigenvalues of $\frac{1}{n}\Delta_n \mathbf{X}_n\mathbf{X}_n^\top \Delta_n$ corresponding to the selected observations. Consequently, $\hat{\sigma}_{n,1}^2$ is well defined on $\{N_n > 0\}$.
\end{remark}

For the test problem 
\begin{align}\label{eq_hypothesis}
    \mathsf{H}_{0,n}: c_n^\top \beta_0 = \theta_{n} \quad \text{vs.} \quad \mathsf{H}_{1,n}: c_n^\top \beta_0 \neq \theta_{n}
\end{align}
we define the test statistic 
\begin{align}\label{eq_test_statistic}
    T_n =\begin{cases}
        0 & \hat{\sigma}_n^2 =0
        \\
        \frac{1}{\hat{\sigma}_n}\frac{\sqrt{r_nn}}{\Vert c_n\Vert}\left(c_n^\top \hat{\beta}-\left(1-a_n(N_n)\right)\theta_n\right) & \hat{\sigma}_n^2 >0
    \end{cases} .
\end{align}

\begin{theorem}\label{thm_test_normal}
    Suppose assumptions \ref{ass_sample_to_dimension_ratio} - \ref{ass_mutual_independence} are satisfied and let $c_n\in \mathbb{R}^{p}\setminus \{\mathbf{0}_p\}$. Furthermore, we assume that $\Vert \beta_0 \Vert = o(1)$ and the null hypothesis in \eqref{eq_hypothesis} holds true. Then,
    \begin{align*}
        T_n \cond \mathcal{N}(0,1).
    \end{align*}
\end{theorem}
Based on Theorem \ref{thm_test_normal}, we propose to reject the null hypothesis $\mathsf{H}_{0,n}$, if
\begin{align*}
    \vert T_n \vert >  q_{1-\frac{\alpha}{2}},
\end{align*}
where $q_{1-{\frac{\alpha}{2}}}$ denotes the $(1-\frac{\alpha}{2})$-quantile of a standard normal distribution for $\alpha \in (0,1)$. 
%The level of the proposed test \eqref{eq_test_statistic} can be controlled asymptotically and is consistent under the alternative $\mathsf{H}_{1,n}$.
In the following corollary, we show that the test \eqref{eq_test_statistic} has asymptotic level $\alpha$ and  is consistent against sequences of alternatives satisfying the separation condition stated below.

\begin{corollary}\label{cor_test_performance}
     Suppose assumptions \ref{ass_sample_to_dimension_ratio} - \ref{ass_mutual_independence} are satisfied. Let $c_n\in \mathbb{R}^{p}\setminus \{\mathbf{0}_p\}$ and assume that $\Vert \beta_0 \Vert = o(1)$. 
     \begin{enumerate}[label=(\alph*)]
        \item Under $\mathsf{H}_{0,n}$ it holds
            \begin{align*}
                \mathbb{P}\left( \vert T_n\vert > q_{1-\frac{\alpha}{2}}\right) \to \alpha, \quad n\to \infty.
            \end{align*}
        \item Suppose that 
        \begin{align*}
            Q_n= \frac{\sqrt{n}}{\Vert c_n\Vert} \vert c_n^\top \beta_0 -\theta_n\vert \to \infty, \quad n\to \infty.
        \end{align*}
        Then, under the alternative $\mathsf{H}_{1,n}$, we have
        \begin{align*}
            \mathbb{P}\left( \vert T_n\vert > q_{1-\frac{\alpha}{2}}\right) \to 1, \quad n\to \infty.
        \end{align*}
        \item 
        Suppose that
        \begin{align*}
            \frac{\sqrt{n}}{\Vert c_n\Vert}(c_n^\top \beta_0 - \theta_n) \to d \in \mathbb{R}.
        \end{align*}
        Then, the test statistic satisfies
        \begin{align*}
            T_n \cond \mathcal{N}(\mu,1),
        \end{align*}
        where
        \begin{align*}
            \mu = \frac{\sqrt{R}}{\sigma} \frac{v\left(\frac{\kappa}{R},-\frac{\tau}{R}\right)}{v\left(\frac{\kappa}{R},-\frac{\tau}{R}\right)+1}d.
        \end{align*}
     \end{enumerate}
\end{corollary}

\subsection{Arbitrary \texorpdfstring{$\bfSigma_n$}{Covariance Matrix}}\label{sec_arbitrary}

In this section, $v(\kappa,z)$ is defined as the solution of
\begin{align}\label{eq_def_general_v}
    v(\kappa,z) = - \left(z-\kappa \int_{0}^{\infty} \frac{t}{1+t v(\kappa,z)}\, dH(t)\right)^{-1}, \quad z\in \mathbb{C}\setminus \mathbb{R}_{\geq 0}
\end{align}
with the Stieltjes-conditions $\Im(v(\kappa,z))\Im(z) >0$, if $z\in \mathbb{C}\setminus\mathbb{R}$ and $v(\kappa, z)>0$, when $z < 0$. Furthermore, $v^\prime(\kappa,z)$ denotes the derivative of $v$ with respect to $z$. Notably, the following central limit theorem for the subsampled ridge estimator does not require the vanishing-signal assumption $\|\beta_0 \| =o(1).$

\begin{theorem}\label{thm_CLT_arbitrary_cov}
    Suppose assumptions \ref{ass_sample_to_dimension_ratio} - \ref{ass_mutual_independence} are satisfied and let $c_n\in \mathbb{R}^p \setminus \{\mathbf{0}_p\}$.
    Assume that either 
    \begin{align}\label{eq_ass_c_norm}
        0 < \inf_{n\in \mathbb{N}} \Vert c_n\Vert \leq \sup_{n\in \mathbb{N}} \Vert c_n\Vert < \infty
    \end{align}
    and
    \begin{align}\label{eq_ass_c_quadratic_form}
        0<\inf_{n\in \mathbb{N}} c_n^\top \bfSigma_n c_n,
    \end{align}
    or that there exists a constant $m>0$ such that
    \begin{align}\label{eq_eigenvalues_bounded_from_below}
        0 < m \leq \lambda_{\min}\left(\bfSigma_n\right)
    \end{align}
    is satisfied, where $\lambda_{\min}(\bfSigma_n)$ denotes the smallest eigenvalue of $\bfSigma_n$. Then
    \begin{align*}
        T_n \cond \mathcal{N}(0,R),
    \end{align*}
    where
    \begin{align*}
        T_n= \begin{cases}
            \frac{1}{\sigma_n}\frac{\sqrt{r_n n}}{\Vert c_n\Vert} c_n^\top \left(\hat{\beta}-\left(\mathbf{I}_p - \tau \mathbf{R}(\mathbf{B}_n, \tau)\right)\beta_0\right) & \sigma_n^2 >0
            \\
            0 & \sigma_n^2=0
        \end{cases},
    \end{align*}
    and
    \begin{align*}
        \sigma_n^2 = \frac{\sigma_{\varepsilon}^2}{\Vert c_n\Vert^2} c_n^\top \left(\mathbf{R}(\mathbf{B}_n, \tau)-\tau \mathbf{R}(\mathbf{B}_n, \tau)^2\right)c_n.
    \end{align*}
\end{theorem}
In the proof of Theorem \ref{thm_consistent_variance_estimator_identity_cov}, we show that $\hat{\sigma}_{n,1}^2 \conp \sigma_{\varepsilon}^2$. We emphasize that the proof does not rely on the assumption $\bfSigma_n=\mathbf{I}_p$ and remains valid for general covariance matrices $\bfSigma_n$ satisfying our assumptions. Consequently, we obtain the following result. 
\begin{corollary}\label{cor_variance_estimator_arbitrary_cov}
Suppose that assumptions \ref{ass_sample_to_dimension_ratio} - \ref{ass_mutual_independence} hold. Let $c_n\in \mathbb{R}^p\setminus\{ \mathbf{0}_p\}$ satisfy either \eqref{eq_ass_c_norm} and \eqref{eq_ass_c_quadratic_form}, or the assumption \ref{eq_eigenvalues_bounded_from_below}. Furthermore, assume that $\Vert\beta_0\Vert = o(1)$. Then,
\begin{align*}
    \frac{\tilde{\sigma}_n^2}{\sigma_n^2} \conp 1,
\end{align*}
where the ratio is defined as $1$ on $\{\sigma_n^2=0\}$ and
\begin{align*}
    \tilde{\sigma}_n^2 = \frac{\hat{\sigma}_{n,1}^2}{\Vert c_n\Vert^2} c_n^\top \left(\mathbf{R}(\mathbf{B}_n,\tau)- \tau \mathbf{R}(\mathbf{B}_n, \tau)^2\right)c_n
\end{align*}
and $\hat{\sigma}_{n,1}^2$ is defined in Theorem \ref{thm_consistent_variance_estimator_identity_cov}.
\end{corollary}

\section{Proofs}
Since $r_n\to R$, there exists a $n_0\in \mathbb{N}$ such that
$r_n > R/2$ for all $n \geq n_0$, which implies that
\begin{align*}
    \mathbb{P}(N_n=0)= \left(1-r_n\right)^n \leq \left(1-\frac{R}{2}\right)^n \to 0.
\end{align*}
Therefore, it suffices to restrict our attention to the event $\{N_n>0\}$ throughout the following proofs when proving convergence in probability or in distribution. 
\newline
For notational convenience, we omit the row index $n$ of the triangular arrays and write $X_i=X_{i,n}$, $\varepsilon_i=\varepsilon_{i,n}$ and $\delta_i = \delta_{i,n}$, whenever no ambiguity arises.

\subsection{Proof of Theorem \ref{thm_main_CLT_identity_cov}}\label{sec_proofs_identity}

We start with  some preparations. First, defining  
$\eta_n = \sqrt{|r_n-R|} + 2\sqrt{\frac{\log n}{n}}$ and noting that  $r_n\to R$, it follows that
\begin{align}\label{eq_def_eta_n}
    \eta_n \to 0, \quad n \eta_n^2 \to \infty, \quad \sqrt{n}\exp\left(-\frac{n\eta_n^2}{2}\right)\to 0, \quad \vert r_n-R\vert=o(\eta_n).
\end{align}
Now, we define
\begin{align*}
    I_n = \{ m \in \{1,\ldots,n\} \mid \vert m - nR \vert \leq n\eta_n\}.
\end{align*}
Then, we obtain
\begin{align*}
    \sup_{m \in I_n} \left\vert \frac{m}{n}-R\right\vert \leq \eta_n \to 0.
\end{align*}
For sufficiently large $n$, we have $\vert r_n- R\vert < \frac{\eta_n}{2}$. Therefore, Hoeffding's inequality implies that
\begin{align}\label{eq_neglible_event}
    \mathbb{P}\left(N_n \notin I_n\right) = \mathbb{P}\left(\left\vert N_n-nR\right\vert > n\eta_n\right) \leq \mathbb{P}\left(\vert N_n - n r_n\vert > \frac{n\eta_n}{2}\right) \leq 2\exp\left(-\frac{n\eta_n^2}{2}\right) \to 0.
\end{align}
With this preparation we are able to prove Theorem \ref{thm_main_CLT_identity_cov}. The approximation \eqref{eq_neglible_event} implies that
    \begin{align*}
        \frac{\sqrt{r_n n}}{\Vert c_n\Vert}c_n^\top \left(\hat{\beta}- (1-a_n(N_n))\beta_0\right) &= \frac{\sqrt{r_n n}}{\Vert c_n\Vert}c_n^\top \left(\hat{\beta}- (1-a_n(N_n))\beta_0\right)\left(\mathds{1}\{N_n \in I_n\}+\mathds{1}\{N_n\notin I_n\}\right)
        \\
        &=\frac{\sqrt{r_n n}}{\Vert c_n\Vert}c_n^\top \left(\hat{\beta}- (1-a_n(N_n))\beta_0\right)\mathds{1}\{N_n \in I_n\} + o_{\mathbb{P}}(1).
    \end{align*}
    By the representation \eqref{eq_rep_ridge_estimator}, we obtain
    \begin{align*}
        &\frac{\sqrt{r_n n}}{\Vert c_n\Vert}c_n^\top \left(\hat{\beta}- (1-a_n(N_n))\beta_0\right)\mathds{1}\{N_n\in I_n\} 
        \\
        =&\frac{\sqrt{r_n n}}{\Vert c_n\Vert}c_n^\top \Bigg( \left[a_n(N_n)\mathbf{I}_p-\tau \mathbf{R}(\mathbf{B}_n,\tau)\right]\beta_0 + \mathbf{R}(\mathbf{B}_n,\tau)\frac{1}{n}\mathbf{X}_n^\top \Delta_n \mathcal{E}_n\Bigg)\mathds{1}\{N_n\in I_n\}
        \\
        =& V_{n,1} + V_{n,2}\mathds{1}\{N_n\in I_n\} ,
    \end{align*}
    where
    \begin{align} \label{eq:def_V_n1}
        V_{n,1} = \frac{\sqrt{r_n n}}{\Vert c_n \Vert} c_n^\top\left[a_n(N_n)\mathbf{I}_p-\tau \mathbf{R}(\mathbf{B}_n,\tau)\right]\beta_0 \mathds{1}\{N_n\in I_n\}
    \end{align}
    and
    \begin{align} \label{eq:def_V_n2}
        V_{n,2} =\frac{\sqrt{r_n n}}{\Vert c_n\Vert }c_n^\top\mathbf{R}(\mathbf{B}_n,\tau)\frac{1}{n}\mathbf{X}_n^\top \Delta_n \mathcal{E}_n.
    \end{align}
We will show below that 
\begin{align}
    V_{n,1} &  \conp 0, \label{eq_bias_term}
    \\
    V_{n,2} &  \cond \mathcal{N}(0,\sigma^2). \label{eq_noise_term} 
\end{align}

From these results the assertion follows.
\medskip

\textit{Proof of \eqref{eq_bias_term}.}
We need to show that
\begin{align}
        &\frac{\sqrt{r_n n}}{\Vert c_n \Vert}\left\vert c_n^\top \left[\tau\mathbf{R}\left(\frac{N_n}{n}\hat{\mathbf{\Gamma}}_n,\tau\right)-\frac{1}{v_n\left(\frac{p}{N_n}, -\frac{n}{N_n}\tau\right) +1}\mathbf{I}_p\right]\beta_0\right\vert\mathds{1}\{N_n \in I_n \} \conp 0,\label{eq_det_equiv_bias_term_high_prob_event_identity_cov}
\end{align}
where $\hat{\mathbf{\Gamma}}_n = \frac{1}{N_n}\sum_{i: \delta_i=1}X_iX_i^\top$.
For any $\varepsilon>0$, we have
    \begin{align}
        \nonumber 
        &\mathbb{P}\left(\left\vert \frac{\sqrt{r_n n}}{\Vert c_n \Vert} c_n^\top \left[\tau\mathbf{R}\left(\frac{N_n}{n}\hat{\mathbf{\Gamma}}_n,\tau\right)-\frac{1}{v\left(\frac{p}{N_n}, -\frac{n}{N_n}\tau\right) +1}\mathbf{I}_p\right]\beta_0\mathds{1}\{N_n \in I_n \}\right\vert > \varepsilon\right)
        \\
            \nonumber 
        =&\sum_{m \in I_n}\mathbb{P}\left(\left\vert \frac{\sqrt{r_n n}}{\Vert c_n \Vert} c_n^\top \left[\tau\mathbf{R}\left(\frac{m}{n}\bfSigmahat_{m,n},\tau\right)-\frac{1}{v\left(\frac{p}{m}, -\frac{n}{m}\tau\right)+1}\mathbf{I}_p\right]\beta_0\right\vert > \varepsilon\right)\mathbb{P}\left(N_n=m\right)
        \\
        \label{det1}
        \leq&\sup_{m \in I_n}\mathbb{P}\left(\left\vert \frac{\sqrt{r_n n}}{\Vert c_n \Vert} c_n^\top \left[\tau\mathbf{R}\left(\frac{m}{n}\bfSigmahat_{m,n},\tau\right)-\frac{1}{v\left(\frac{p}{m}, -\frac{n}{m}\tau\right)+1}\mathbf{I}_p\right]\beta_0\right\vert > \varepsilon\right),
    \end{align}
    where we used that
    \begin{align*}
        \hat{\mathbf{\Gamma}}_n \mid \{ N_n=m \} \sim \bfSigmahat_{m,n}
    \end{align*}
    with
    \begin{align*}
        \bfSigmahat_{m,n} = \frac{1}{m}\sum_{i=1}^{m} X_{i,n} X_{i,n}^\top
    \end{align*}
    because $X_{1,n}, \ldots, X_{n,n}$ are i.i.d.. Consequently, \eqref{eq_det_equiv_bias_term_high_prob_event_identity_cov} follows, if the right hand side of \eqref{det1} converges to $0$.
    Suppose that this is not the case.
    Then there exists a $\delta > 0$ and a subsequence $n_k$ such that
    \begin{align*}
        \sup_{m\in I_{n_k}}\mathbb{P}\left(\frac{\sqrt{r_{n_k}n_k}}{\Vert c_{n_k} \Vert}\left\vert c_{n_k}^\top \left[\tau\mathbf{R}\left(\frac{m}{n_k}\bfSigmahat_{m,n_k},\tau\right)-\frac{1}{v\left(\frac{p}{m}, -\frac{n_k}{m}\tau\right)+1}\mathbf{I}_p\right]\beta_0\right\vert>\varepsilon\right) > \delta.
    \end{align*}
    This implies that there exists a $m_k\in I_{n_k}$ such that
    \begin{align}\label{eq_contraposition_det_equiv_bias_term_identity_cov}
        \mathbb{P}\left(\frac{\sqrt{r_{n_k}n_k}}{\Vert c_{n_k} \Vert}\left\vert c_{n_k}^\top \left[\tau\mathbf{R}\left(\frac{m_k}{n_k}\bfSigmahat_{m_k,n_k},\tau\right)-\frac{1}{v\left(\frac{p}{m_k}, -\frac{n_k}{m_k}\tau\right)+1}\mathbf{I}_p\right]\beta_0\right\vert>\varepsilon\right) > \delta.
    \end{align}
    Because $m_k\in I_{n_k}$ for all $k\in \mathbb{N}$, we obtain
    \begin{align*}
        \left\vert \frac{m_k}{n_k}-R\right\vert \leq \eta_{n_k} \to 0,
    \end{align*}
    which implies 
    \begin{align*}
        \frac{p}{m_k} = \frac{p/n_k}{m_k/n_k}\to \frac{\kappa}{R}.
    \end{align*}
    Consequently, there exists a compact set $K\subset \mathbb{C}\setminus \mathbb{R}_{\geq 0}$ such that $\zeta_k = -\frac{n_k}{m_k}\tau \in K$ for sufficiently large $k$.
    Since 
    \begin{align*}
        \left\Vert \frac{\sqrt{r_{n_k}n_{k}}}{\Vert c_{n_k}\Vert} \tau \beta_0 c_{n_k}^\top \right\Vert_{F}= \sqrt{r_{n_k}n_{k}}\tau \Vert \beta_0\Vert,
    \end{align*}
    it follows by Lemma \ref{lem_det_equi_DS_identity_cov} that
    \begin{align*}
        & \frac{\sqrt{r_{n_k}n_k}}{\Vert c_{n_k}\Vert} \frac{m_k}{n_k} c_{n_k}^\top \left[\mathbf{R}\left(\frac{m_k}{n_k}\bfSigmahat_{m_k,n_k},\tau\right) - \frac{1}{\tau v\left(\frac{p}{m_k}, -\frac{n_k}{m_k}\tau\right)+\tau}\mathbf{I}_p\right]\beta_0 
        \\
        =&\frac{\sqrt{r_{n_k}n_k}}{\Vert c_{n_k}\Vert} c_{n_k}^\top \left[\mathbf{R}\left(\bfSigmahat_{m_k,n_k},\frac{n_k}{m_k}\tau\right)-\frac{1}{\frac{n_k}{m_k}\tau v\left(\frac{p}{m_k},-\frac{n_k}{m_k}\tau\right)+\frac{n_k}{m_k}\tau}\mathbf{I}_p\right]\beta_0
        \\
        =&O_{\mathbb{P}}\left(\sqrt{r_{n_k}}\tau \Vert \beta_0\Vert + \sqrt{\frac{r_{n_k}}{n_k}}\tau \Vert \beta_0\Vert\right) = o_{\mathbb{P}}(1),
    \end{align*}
    where we used $\Vert \beta_0\Vert=o(1)$.
    Because $m_k/n_k \to R$, we obtain 
    \begin{align*}
         \frac{\sqrt{r_{n_k}n_k}}{\Vert c_{n_k}\Vert} c_{n_k}^\top \left[\tau\mathbf{R}\left(\frac{m_k}{n_k}\bfSigmahat_{m_k,n_k},\tau\right) - \frac{1}{ v\left(\frac{p}{m_k}, -\frac{n_k}{m_k}\tau\right)+1}\mathbf{I}_p\right]\beta_0 \conp 0.
    \end{align*}
    This is a contradiction to \eqref{eq_contraposition_det_equiv_bias_term_identity_cov}, and therefore the right hand side of \eqref{det1} converges to $0$.

    \smallskip

    \textit{Proof of \eqref{eq_noise_term}.}
    It suffices to show the conditional CLT
    \begin{align*}
        \frac{\sqrt{r_n}}{\Vert c_n\Vert\sqrt{n}} c_n^\top \mathbf{R}\left(\mathbf{B}_n,\tau\right)\mathbf{X}_n^\top \Delta_n \mathcal{E}_n \mid \mathbf{X}_n, \Delta_n \cond \mathcal{N}(0, \sigma^2).
    \end{align*}
    Because $r_n\to R$, we show that
    \begin{align*}
        \frac{1}{\Vert c_n \Vert\sqrt{n}}c_n^\top \mathbf{R}\left(\mathbf{B}_n,\tau\right)\mathbf{X}_n^\top \Delta_n \mathcal{E}_n \mid \mathbf{X}_n, \Delta_n
    \end{align*}
    satisfies the (conditional) Ljapunov condition. In particular, we show
    \begin{align*}
        \frac{1}{s_n^4}\sum_{i=1}^{n}\mathbb{E}\left[\left(\frac{1}{\Vert c_n \Vert\sqrt{n}}c_n^\top\mathbf{R}\left(\mathbf{B}_n,\tau\right)X_i\delta_i\varepsilon_i \right)^4\mid \mathbf{X}_n, \Delta_n\right] \conp 0,
    \end{align*}
    where 
    \begin{align*}
        s_n^2 &= \sum_{i=1}^{n} \Var \left(\frac{1}{\Vert c_n \Vert\sqrt{n}}c_n^\top\mathbf{R}\left(\mathbf{B}_n,\tau\right)X_i\delta_i\varepsilon_i \mid \mathbf{X}_n, \Delta_n\right)= \frac{\sigma_{\varepsilon}^2}{\Vert c_n\Vert^2n} \sum_{i=1}^{n} \left(c_n^\top\mathbf{R}\left(\mathbf{B}_n,\tau\right)X_i\delta_i\right)^2
        \\
        &= \frac{\sigma_{\varepsilon}^2}{\Vert c_n\Vert^2n}\sum_{i=1}^{n}c_n^\top\mathbf{R}\left(\mathbf{B}_n,\tau\right)X_i\delta_i^2 X_i^\top\mathbf{R}\left(\mathbf{B}_n,\tau\right)c_n = \frac{\sigma_{\varepsilon}^2}{\Vert c_n\Vert^2}c_n^\top \mathbf{R}\left(\mathbf{B}_n,\tau\right) \mathbf{B}_n \mathbf{R}\left(\mathbf{B}_n,\tau\right) c_n
        \\
        &=\frac{\sigma_{\varepsilon}^2}{\Vert c_n\Vert^2} c_n^\top \left(\mathbf{R}\left(\mathbf{B}_n,\tau\right) - \tau \mathbf{R}\left(\mathbf{B}_n,\tau\right)^{2}\right)c_n = \sigma_{\varepsilon}^2\hat{\sigma}_{n,2}^{2}.
\end{align*}
In the proof of Theorem \ref{thm_consistent_variance_estimator_identity_cov}, we show that 
\begin{align*}
    \hat{\sigma}_{n,2}^2 \conp \frac{1}{R}\frac{v^\prime\left(\frac{\kappa}{R},-\frac{\tau}{R}\right)}{\left(v\left(\frac{\kappa}{R},-\frac{\tau}{R}\right)+1\right)^2} >0,
\end{align*}
which implies that $s_n^2 \conp \sigma^2/ R$. Thus, it remains to show that
\begin{align}
    \sum_{i=1}^{n}\mathbb{E}\left[\left(\frac{1}{\Vert c_n \Vert\sqrt{n}}c_n^\top\mathbf{R}\left(\mathbf{B}_n,\tau\right)X_i\delta_i\varepsilon_i \right)^4\mid \mathbf{X}_n, \Delta_n\right] \conp 0.\label{eq_cond_ljapunov_zero_id_cov}
\end{align}
    
    Recall that $\varepsilon_1, \ldots, \varepsilon_n$ are independent of $\mathbf{X}_n$ and $\Delta_n$ by assumption \ref{ass_mutual_independence}. Hence, we obtain
    \begin{align}
        \sum_{i=1}^{n}\mathbb{E}\left[\left(\frac{1}{\Vert c_n \Vert\sqrt{n}}c_n^\top\mathbf{R}\left(\mathbf{B}_n,\tau\right)X_i\delta_i\varepsilon_i \right)^4\mid \mathbf{X}_n, \Delta_n\right]
        =& \frac{1}{\Vert c_n\Vert^4}\frac{1}{n^2}\sum_{i=1}^{n}\left(c_n^\top\mathbf{R}\left(\mathbf{B}_n,\tau\right)X_i\delta_i\right)^4 \mathbb{E}\left[\varepsilon_i^4\right]\nonumber
        \\
        \lesssim& \frac{1}{\Vert c_n\Vert^4}\frac{1}{n^2}\sum_{i=1}^{n}\left(c_n^\top\mathbf{R}\left(\mathbf{B}_n,\tau\right)X_i\delta_i\right)^4 \nonumber
        \\
        \lesssim& \frac{1}{\Vert c_n\Vert^4}\frac{1}{n^2}\sum_{i=1}^{n}\left(c_n^\top\mathbf{R}\left(\mathbf{B}_n,\tau\right)X_i\right)^4,\label{eq_cond_approx_Ljapunov}
    \end{align}
    where we used assumption \ref{ass_random_errors}.
    For any $i=1, \ldots, n,$ the Sherman-Woodbury formula implies that
    \begin{align*}
        \mathbf{R}(\mathbf{B}_n, \tau) = \mathbf{R}(\mathbf{B}_n^{(i)}, \tau)- \frac{\mathbf{R}(\mathbf{B}_n^{(i)}, \tau)\frac{1}{n}X_iX_i^\top\delta_i \mathbf{R}(\mathbf{B}_n^{(i)}, \tau)}{1+ \frac{1}{n}\delta_i X_i^\top \mathbf{R}(\mathbf{B}_n^{(i)},\tau)X_i},
    \end{align*}
    where 
    \begin{align*}
        \mathbf{B}_n^{(i)} &= \mathbf{B}_n - \frac{1}{n}X_iX_i^\top \delta_i.
    \end{align*}
    By multiplying with $c_n^\top$ from left and $X_i\delta_i$ from right, we obtain
    \begin{align*}
        c_n^\top \mathbf{R}(\mathbf{B}_n,\tau)X_i\delta_i = \frac{c_n^\top \mathbf{R}(\mathbf{B}_n^{(i)},\tau)X_i\delta_i}{1+ \frac{1}{n}\delta_i X_i^\top \mathbf{R}(\mathbf{B}_n^{(i)},\tau)X_i}.
    \end{align*}
    Since $\mathbf{B}_n$ is nonnegative definite, it follows that
    \begin{align}\label{eq_Sherman_Woodbury_id_cov}
        \left\vert c_n^\top \mathbf{R}(\mathbf{B}_n,\tau) X_i \delta_i \right\vert \leq \left\vert c_n^\top \mathbf{R}(\mathbf{B}_n^{(i)},\tau) X_i \delta_i\right\vert \leq \left\vert c_n^\top \mathbf{R}(\mathbf{B}_n^{(i)},\tau) X_i \right\vert.
    \end{align}
    Because $X_i$ is independent of $\mathbf{B}_n^{(i)}$, we have $\mathbb{E}[c_n^\top \mathbf{R}(\mathbf{B}_n^{(i)},\tau)e_j X_i(j)\mid \mathbf{B}_n^{(i)}]=0$ for all $j=1, \ldots, p$, where $e_j$ is the $j$-th standard unit vector and $X_i(j)$ is the $j$-th component of $X_i$. Now, recall that $\bfSigma_n=\mathbf{I}_p$, which implies that the entries of $X_i$ are independent. Therefore, conditional on $\mathbf{B}_n^{(i)}$ the random variables $c_n^\top \mathbf{R}(\mathbf{B}_n^{(i)},\tau)e_jX_i(j)$ are independent. Consequently, the (conditional) Rosenthal inequality yields 
    \begin{align*}
        \mathbb{E}\left[\left\vert c_n^\top \mathbf{R}(\mathbf{B}_n,\tau)X_i\right\vert^4\right]&\leq\mathbb{E}\left[\left\vert c_n^\top \mathbf{R}(\mathbf{B}_n^{(i)},\tau)X_i\right\vert^4\right] = \mathbb{E}\left[\mathbb{E}\left[\left\vert\sum_{j=1}^{p}c_n^\top \mathbf{R}(\mathbf{B}_n^{(i)},\tau)e_j X_i(j)\right\vert^4\mid \mathbf{B}_{n}^{(i)}\right]\right]
        \\
        &\lesssim \mathbb{E}\left[\mathbb{E}\left[\sum_{j=1}^{p}\left\vert c_n^\top \mathbf{R}(\mathbf{B}_n^{(i)},\tau)e_j X_i(j)\right\vert^4\mid \mathbf{B}_{n}^{(i)}\right]\right] + \mathbb{E}\left[\mathbb{E}\left[\sum_{j=1}^{p} \left\vert c_n^\top \mathbf{R}(\mathbf{B}_n^{(i)},\tau)e_jX_i(j)\right\vert^2\mid \mathbf{B}_{n}^{(i)}\right]^2\right]
        \\
        &\lesssim \mathbb{E}\left[\sum_{j=1}^{p} \left\vert c_n^\top \mathbf{R}(\mathbf{B}_n^{(i)},\tau) e_j \right\vert^4 \right] + \mathbb{E}\left[\left(\sum_{j=1}^{p} \left\vert c_n^\top \mathbf{R}(\mathbf{B}_n^{(i)},\tau)e_j\right\vert^2\right)^2\right]
        \\
        &\leq \mathbb{E}\left[\left\Vert \mathbf{R}(\mathbf{B}_n^{(i)},\tau)c_n\right\Vert^4\right] + \mathbb{E}\left[\left\Vert \mathbf{R}(\mathbf{B}_n^{(i)},\tau)c_n\right\Vert^4\right]
        \\
        &\lesssim \Vert c_n \Vert^4.
    \end{align*}
    Consequently, we have
    \begin{align}
        \frac{1}{\Vert c_n \Vert^4} \frac{1}{n^2}\sum_{i=1}^{n} \mathbb{E}\left[\left(c_n^\top \mathbf{R}(\mathbf{B}_n,\tau) X_i\right)^4\right] &\leq \frac{1}{\Vert c_n\Vert^4} \frac{1}{n^2}\sum_{i=1}^{n} \mathbb{E}\left[\left(c_n^\top \mathbf{R}(\mathbf{B}_n^{(i)},\tau)X_i\right)^4\right]\nonumber
        \\
        &\lesssim \frac{1}{\Vert c_n\Vert^4}\frac{1}{n^2}\sum_{i=1}^{n} \Vert c_n\Vert^4 = \frac{1}{n} = o(1). \label{eq_expectation_vanish_Ljapunov}
    \end{align}
    Finally, by Markov's inequality with the bounds \eqref{eq_cond_approx_Ljapunov} and \eqref{eq_expectation_vanish_Ljapunov}, we conclude
    \eqref{eq_cond_ljapunov_zero_id_cov}.

\subsection{Proof of  Theorem \ref{thm_consistent_variance_estimator_identity_cov}}
We will show below that
\begin{align}
    \hat{\sigma}_{n,1}^2 &\conp \sigma_{\varepsilon}^2, \label{eq_consistent_noise_variance}
    \\
    \hat{\sigma}_{n,2}^2 & \conp \frac{\frac{1}{R}v^\prime\left(\frac{\kappa}{R},-\frac{\tau}{R}\right)}{\left(v\left(\frac{\kappa}{R}, -\frac{\tau}{R}\right)+1\right)^2} > 0. \label{eq_asymptotic_variance}
\end{align}
Combining these results with $r_n \to R$, we obtain the assertion.

\medskip

\textit{Proof of \eqref{eq_consistent_noise_variance}.}
By the definition of $\mathcal{Y}_n$, we obtain
    \begin{align*}
        \frac{1}{n}\mathcal{Y}_n^\top \left(\Delta_n - \mathbf{M}_n\right) \mathcal{Y}_n &= \frac{1}{n}\mathcal{E}_n^\top \left(\Delta_n -\mathbf{M}_n\right) \mathcal{E}_n + \frac{1}{n}\beta_0^\top \mathbf{X}_n^\top  \left(\Delta_n -\mathbf{M}_n\right) \mathbf{X}_n\beta_0
        \\
        &\phantom{=}~+\frac{2}{n}\beta_0^\top \mathbf{X}_n^\top \left(\Delta_n-\mathbf{M}_n\right) \mathcal{E}_n
        \\
        &=: Q_{n,1} + Q_{n,2} + Q_{n,3}.
    \end{align*}
    We show in several steps
    \begin{align}
        Q_{n,1} &= \left(\frac{N_n}{n}-\frac{1}{n}\tr\left(\mathbf{M}_n\right)\right) \sigma_{\varepsilon}^2 + o_{\mathbb{P}}(1),\label{eq_consistent_estimator_quad_1}
        \\
        Q_{n,2} &\conp 0,\label{eq_consistent_estimator_quad_2}
        \\
        Q_{n,3} &\conp 0.\label{eq_consistent_estimator_quad_3}
    \end{align}
    Combining these results with Lemma \ref{lem_demoninator_var_est_away_from_zero} completes the proof of the assertion.
    To show \eqref{eq_consistent_estimator_quad_1}, we first note that
    \begin{align*}
        \mathbb{E}\left[\frac{1}{n}\mathcal{E}_n^\top \Delta_n \mathcal{E}_n\right] = r_n \sigma_{\varepsilon}^2
    \end{align*}
    and
    \begin{align*}
        \Var\left(\frac{1}{n}\mathcal{E}_n^\top \Delta_n \mathcal{E}_n\right) = \frac{1}{n}\Var(\delta_{1}\varepsilon_1^2)\leq \frac{r_n}{n}\mathbb{E}[\varepsilon_1^4] = O(n^{-1}).
    \end{align*}
    This implies that 
    \begin{align}\label{eq_Q_1_n_part_1}
        \frac{1}{n}\mathcal{E}_n^\top \Delta_n\mathcal{E}_n = r_n\sigma_{\varepsilon}^2 + O_{\mathbb{P}}\left(n^{-1/2}\right).
    \end{align}
    Moreover, we have
    \begin{align*}
    \Var\left(\mathcal{E}_n^\top \mathbf{M}_n\mathcal{E}_n\mid \mathbf{X}_n, \Delta_n\right)&=\mathbb{E}\left[\left(\mathcal{E}_n^\top \mathbf{M}_n \mathcal{E}_n-\sigma_{\varepsilon}^2\tr\left(\mathbf{M}_n\right)\right)^2\mid \mathbf{X}_n, \Delta_n\right]
    \\
    &=\left(\mathbb{E}\left[\varepsilon_1^4\right]-\sigma_{\varepsilon}^4\right)\sum_{j=1}^{n}\mathbf{M}_n(j,j)^2 + 2\sigma_{\varepsilon}^4\sum_{j\neq \ell}^{n} \mathbf{M}_n(j,\ell)^2
    \\
    &\leq \left(\mathbb{E}\left[\varepsilon_1^4\right]+\sigma_{\varepsilon}^4\right) \left\Vert \mathbf{M}_n\right\Vert_F^2.
\end{align*}
Since $0\preceq \mathbf{M}_n\preceq\Delta_n\preceq I_n,$ 
we have $ \Vert \mathbf{M}_n\Vert_F^2\leq n$, which yields 
\begin{align*}
    \Var\left(\mathcal{E}_n^\top \mathbf{M}_n\mathcal{E}_n\mid X_n,\Delta_n\right)\leq n\left(\mathbb{E}\left[\varepsilon_1^4\right]+\sigma_\varepsilon^4\right).
\end{align*}
By Chebyshev's inequality, this implies
\begin{align}\label{eq_Q_1_n_part_2}
    \frac{1}{n}\mathcal{E}_n^\top \mathbf{M}_n \mathcal{E}_n = \frac{\sigma_{\varepsilon}^2}{n}\tr \left(\mathbf{M}_n\right) + O_{\mathbb{P}}(n^{-1/2}).
\end{align}

Combining \eqref{eq_Q_1_n_part_1} and \eqref{eq_Q_1_n_part_2}, we obtain
\begin{align*}
    \frac{1}{n}\mathcal{E}_n^\top \left(\Delta_n -\mathbf{M}_n\right)\mathcal{E}_n &= \left(r_n-\frac{1}{n}\tr\left(\mathbf{M}_n\right)\right)\sigma_{\varepsilon}^2 + O_{\mathbb{P}}\left(n^{-1/2}\right)
    \\
    &= \left(\frac{N_n}{n}- \frac{1}{n}\tr(\mathbf{M}_n)\right)\sigma_{\varepsilon}^2 + o_{\mathbb{P}}(1),
\end{align*}
where we used \eqref{eq_subsampled_ratio_difference_as_convergence}. 

\smallskip

By the definition of $\mathbf{M}_n$, we obtain 
    \begin{align*}
        Q_{n,2} &= \beta_0^\top \mathbf{B}_n \beta_0 - \frac{1}{n}\beta_0^\top\mathbf{X}_n^\top \mathbf{M}_n \mathbf{X}_n \beta_0
        \\
        &= \beta_0^\top \mathbf{B}_n \beta_0 - \beta_0^\top\mathbf{B}_n \left(\mathbf{I}_p -\tau \mathbf{R}(\mathbf{B}_n,\tau)\right)\beta_0
        \\
        &= \tau\beta_0^\top\mathbf{B}_n\mathbf{R}(\mathbf{B}_n,\tau)\beta_0
        \\
        &= \tau\Vert \beta_0 \Vert ^2 - \tau^2 \beta_0^\top\mathbf{R}(\mathbf{B}_n,\tau)\beta_0.
    \end{align*}
    This leads to
    \begin{align*}
        \vert Q_{n,2}\vert \leq 2\tau \Vert \beta_0\Vert^2 =o(1),
    \end{align*}
    because $\Vert \beta_0\Vert =o(1)$. Hence, \eqref{eq_consistent_estimator_quad_2} follows.

    \smallskip

    For \eqref{eq_consistent_estimator_quad_3}, we notice that
    \begin{align*}
        \mathbb{E}\left[Q_{n,3}\mid \mathbf{X}_n, \Delta_n\right] =0,
    \end{align*}
    due to the mutual independence of $\mathcal{E}_n, \mathbf{X}_n$ and $\Delta_n$ and the fact that $\mathbb{E}\left[\mathcal{E}_n\right]=\mathbf{0}_n$. 
    Furthermore, we have
    \begin{align*}
        \Var\left(Q_{n,3} \mid \mathbf{X}_n, \Delta_n\right) &= \mathbb{E}\left[Q_{n,3}^2 \mid \mathbf{X}_n, \Delta_n\right]
        \\
        &=\frac{4}{n^2}\mathbb{E}\left[\beta_0^\top \mathbf{X}_n^\top \left(\Delta_n - \mathbf{M}_n\right)\mathcal{E}_n \mathcal{E}_n^\top \left(\Delta_n -\mathbf{M}_n\right)\mathbf{X}_n\beta_0\mid \mathbf{X}_n, \Delta_n\right]
        \\
        &= \frac{4\sigma_{\varepsilon}^2}{n^2} \beta_0^\top \mathbf{X}_n^\top \left(\Delta_n -\mathbf{M}_n\right)^2 \mathbf{X}_n \beta_0
        \\
        &= \frac{4\sigma_{\varepsilon}^2}{n^2}\beta_0^\top \left(\mathbf{X}_n^\top \Delta_n \mathbf{X}_n - \mathbf{X}_n^\top \Delta_n \mathbf{M}_n \mathbf{X}_n - \mathbf{X}_n^\top \mathbf{M}_n \Delta_n \mathbf{X}_n + \mathbf{X}_n^\top \mathbf{M}_n^2 \mathbf{X}_n\right)\beta_0
        \\
        &=\frac{4\sigma_{\varepsilon}^2}{n^2}\beta_0^\top \left(\mathbf{X}_n^\top \Delta_n \mathbf{X}_n - 2\mathbf{X}_n^\top \mathbf{M}_n \mathbf{X}_n + \mathbf{X}_n^\top \mathbf{M}_n^2 \mathbf{X}_n\right)\beta_0
        \\
        &= \frac{4\sigma_{\varepsilon}^2}{n}\beta_0^\top \bigg(\mathbf{B}_n - 2\mathbf{B}_n\left(\mathbf{I}_p-\tau \mathbf{R}(\mathbf{B}_n,\tau)\right)
        \\
        &\phantom{=\frac{4\sigma_{\varepsilon}^2}{n^2}\beta_0^\top ~~~}+ \mathbf{B}_n\left(\mathbf{I}_p - 2\tau \mathbf{R}(\mathbf{B}_n,\tau)+ \tau^2 \mathbf{R}(\mathbf{B}_n,\tau)^2\right)\bigg)\beta_0
        \\
        &= \frac{4\sigma_{\varepsilon}^2}{n}\beta_0^\top \left(\tau^2\mathbf{B}_n\mathbf{R}(\mathbf{B}_n,\tau)^2\right)\beta_0
        \\
        &= \frac{4\sigma_{\varepsilon}^2}{n}\tau^2\beta_0^\top\left(\mathbf{R}(\mathbf{B}_n,\tau)-\tau \mathbf{R}(\mathbf{B}_n,\tau)^2\right)\beta_0
        \\
        &\leq \frac{4\sigma_{\varepsilon}^2}{n}\tau^2 \frac{2}{\tau}\Vert \beta_0\Vert^2 = \frac{8\sigma_{\varepsilon}^2}{n}\tau\Vert \beta_0\Vert^2 = o(1),
    \end{align*}
    where we used $\Vert \beta_0\Vert=o(1)$ in the last line. Consequently, we have shown \eqref{eq_consistent_estimator_quad_3}.

\medskip

\textit{Proof of \eqref{eq_asymptotic_variance}.}
Because of \eqref{eq_neglible_event}, it suffices to show that
\begin{align}\label{eq_asymp_variance_first_quad_form_identity_cov}
        \frac{1}{\Vert c_n \Vert^2}\left\vert c_n^\top\left[\mathbf{R}\left(\frac{N_n}{n}\hat{\mathbf{\Gamma}}_n,\tau\right)-\frac{1}{\tau v(\frac{\kappa}{R},-\frac{\tau}{R})+ \tau}\mathbf{I}_p\right]c_n\right\vert \mathds{1}\left\{N_n\in I_n\right\}\conp 0,
    \end{align}
    and
    \begin{align}\label{eq_asymp_variance_second_quad_form_identity_cov}
        \frac{1}{\Vert c_n\Vert^2} \left\vert c_n^\top \left[\tau \mathbf{R}\left(\frac{N_n}{n}\hat{\mathbf{\Gamma}}_n,\tau\right)^2- \frac{v(\frac{\kappa}{R}, -\frac{\tau}{R})-\frac{1}{R}\tau v^\prime(\frac{\kappa}{R},-\frac{\tau}{R}) +1}{\tau\left(v(\frac{\kappa}{R}, -\frac{\tau}{R})+1\right)^2}\mathbf{I}_p\right]c_n\right\vert \mathds{1}\{N_n \in I_n\} \conp 0,
    \end{align}
    where $\hat{\mathbf{\Gamma}}_n = \frac{1}{N_n}\sum_{i: \delta_i=1}X_iX_i^\top$. We start with \eqref{eq_asymp_variance_first_quad_form_identity_cov}. For any $\varepsilon>0$, we have
    \begin{align}
        &\mathbb{P}\left( \frac{1}{\Vert c_n \Vert^2}\left\vert c_n^\top\left[\mathbf{R}\left(\frac{N_n}{n}\hat{\mathbf{\Gamma}}_n,\tau\right)-\frac{1}{\tau v(\frac{\kappa}{R},-\frac{\tau}{R})+ \tau}\mathbf{I}_p\right]c_n\right\vert \mathds{1}\left\{N_n\in I_n\right\} > \varepsilon\right)\nonumber
        \\
        =& \sum_{m\in I_n}\mathbb{P}\left(\frac{1}{\Vert c_n \Vert^2}\left\vert c_n^\top \left[\mathbf{R}\left(\frac{m}{n}\bfSigmahat_{m,n},\tau\right) - \frac{1}{\tau v(\frac{\kappa}{R}, - \frac{\tau}{R})+ \tau}\mathbf{I}_p\right] c_n\right\vert > \varepsilon\right) \mathbb{P}\left(N_n=m\right)\nonumber
        \\
        \leq & \sup_{m\in I_n}\mathbb{P}\left(\frac{1}{\Vert c_n \Vert^2}\left\vert c_n^\top \left[\mathbf{R}\left(\frac{m}{n}\bfSigmahat_{m,n},\tau\right) - \frac{1}{\tau v(\frac{\kappa}{R}, - \frac{\tau}{R})+ \tau}\mathbf{I}_p\right] c_n\right\vert > \varepsilon\right), \label{det2}
    \end{align}
    where we used that
    \begin{align*}
        \hat{\mathbf{\Gamma}}_n \mid \{ N_n=m \} \sim \bfSigmahat_{m,n},
    \end{align*}
    because $X_{1,n}, \ldots, X_{n,n}$ are i.i.d.. Consequently, \eqref{eq_asymp_variance_first_quad_form_identity_cov} follows, if the right hand side of \eqref{det2} converges to $0$. Suppose that this is not the case. Then, there exists a $\delta >0$ and a subsequence     $n_k$ such that 
    \begin{align*}
        \sup_{m\in I_{n_k}}\mathbb{P}\left(\frac{1}{\Vert c_{n_k} \Vert^2}\left\vert c_{n_k}^\top \left[\mathbf{R}\left(\frac{m}{n_k}\bfSigmahat_{m,n_k},\tau\right) - \frac{1}{\tau v(\frac{\kappa}{R}, - \frac{\tau}{R})+ \tau}\mathbf{I}_p\right] c_{n_k}\right\vert > \varepsilon\right)> \delta.
    \end{align*}
    This implies that there exists a $m_k \in I_{n_k}$ such that
    \begin{align}\label{eq_contradiction_subsequence_identity_cov}
        \mathbb{P}\left(\frac{1}{\Vert c_{n_k} \Vert^2}\left\vert c_{n_k}^\top \left[\mathbf{R}\left(\frac{m_k}{n_k}\bfSigmahat_{m_k,n_k},\tau\right) - \frac{1}{\tau v(\frac{\kappa}{R}, - \frac{\tau}{R})+ \tau}\mathbf{I}_p\right] c_{n_k}\right\vert > \varepsilon\right)> \delta.
    \end{align}
    Because $m_k \in I_{n_k}$ for all $k\in \mathbb{N}$, we obtain
    \begin{align*}
        \left\vert \frac{m_k}{n_k}-R \right\vert \leq \eta_{n_k}\to 0,
    \end{align*}
    which implies
    \begin{align*}
        \frac{p}{m_k} = \frac{p/n_k}{m_k/n_k} \to \frac{\kappa}{R}.
    \end{align*}
    Consequently, there exists a compact set $K \subset \mathbb{C}\setminus \mathbb{R}_{\geq 0}$ such that $\zeta_k= -\frac{n_k}{m_k}\tau \in K$ for sufficiently large $k$.
    Since $\Vert c_{n_k}c_{n_k}^\top \Vert_{\tr}/\Vert c_{n_k}\Vert^2=1$, it follows by Lemma \ref{lem_det_equi_identity_cov} that
    \begin{align*}
        &\frac{1}{\Vert c_{n_k} \Vert^2}c_{n_k}^\top \left[\mathbf{R}\left(\bfSigmahat_{m_k,n_k}, \frac{n_k}{m_k}\tau\right) - \frac{1}{\frac{n_k}{m_k}\tau v(\frac{\kappa}{R}, - \frac{n_k}{m_k}\tau) + \frac{n_k}{m_k}\tau}\mathbf{I}_p\right]c_{n_k}
        \\
        =&\frac{1}{\Vert c_{n_k} \Vert^2}\frac{m_k}{n_k}c_{n_k}^\top \left[\mathbf{R}\left(\frac{m_k}{n_k}\bfSigmahat_{m_k,n_k}, \tau\right)- \frac{1}{\tau v(\frac{\kappa}{R},-\frac{n_k}{m_k}\tau)+ \tau}\mathbf{I}_p\right]c_{n_k} \conp 0.
    \end{align*}
    Because $m_k/n_k \to R$, we obtain
    \begin{align*}
        \frac{1}{\Vert c_{n_k} \Vert^2}c_{n_k}^\top \left[\mathbf{R}\left(\frac{m_k}{n_k}\bfSigmahat_{m_k,n_k}, \tau\right)- \frac{1}{\tau v(\frac{\kappa}{R},-\frac{n_k}{m_k}\tau)+ \tau}\mathbf{I}_p\right]c_{n_k} \conp 0.
    \end{align*}
    Recall that $v(\kappa,z)$ denotes the Stieltjes transform of the companion Marchenko-Pastur distribution $\underline{F}^{\kappa, \delta_{\{1\}}}$. Therefore, $v(\kappa,z)$ is holomorphic on $\mathbb{C}\setminus \mathbb{R}_{\geq 0}\subset\mathbb{C}\setminus \supp\left(\underline{F}^{\kappa, \delta_{\{1\}}}\right)$ in $z$. Hence, we obtain
    \begin{align*}
        \frac{1}{\Vert c_{n_k} \Vert^2}\left\vert c_{n_k}^\top \left[\left(\frac{1}{\tau v(\frac{\kappa}{R}, -\frac{n_k}{m_k}\tau)+ \tau} - \frac{1}{\tau v(\frac{\kappa}{R}, -\frac{\tau}{R})+ \tau}\right)\mathbf{I}_p\right]c_{n_k}\right\vert \leq \frac{1}{\tau} \left\vert v\left(\frac{\kappa}{R}, - \frac{\tau}{R}\right) - v\left(\frac{\kappa}{R}, - \frac{n_k}{m_k}\tau\right)\right\vert \to 0,
    \end{align*}
    which implies that 
    \begin{align*}
        \frac{1}{\Vert c_{n_k} \Vert^2}c_{n_k}^\top \left[\mathbf{R}\left(\frac{m_k}{n_k}\bfSigmahat_{m_k,n_k}, \tau\right)- \frac{1}{\tau v(\frac{\kappa}{R},-\frac{\tau}{R})+ \tau}\mathbf{I}_p\right]c_{n_k} \conp 0.
    \end{align*}
    This is a contradiction to \eqref{eq_contradiction_subsequence_identity_cov}, and therefore the right hand side of \eqref{det2} converges to $0$.
    
    \smallskip
    
    Next, we prove \eqref{eq_asymp_variance_second_quad_form_identity_cov}.
    For any $\varepsilon >0$, we have
    \begin{align}
        &\mathbb{P}\left(\frac{1}{\Vert c_n\Vert^2} \left\vert c_n^\top \left[\tau \mathbf{R}\left(\frac{N_n}{n}\hat{\mathbf{\Gamma}}_n,\tau\right)^2- \frac{v(\frac{\kappa}{R}, -\frac{\tau}{R})-\frac{1}{R}\tau v^\prime(\frac{\kappa}{R},-\frac{\tau}{R}) +1}{\tau\left(v(\frac{\kappa}{R}, -\frac{\tau}{R})+1\right)^2}\mathbf{I}_p\right]c_n\right\vert \mathds{1}\{N_n \in I_n\} > \varepsilon \right)\nonumber
        \\
        =& \sum_{m\in I_n}\mathbb{P}\left(\frac{1}{\Vert c_n\Vert^2} \left\vert c_n^\top \left[\tau \mathbf{R}\left(\frac{m}{n}\bfSigmahat_{m,n},\tau\right)^2- \frac{v(\frac{\kappa}{R}, -\frac{\tau}{R})-\frac{1}{R}\tau v^\prime(\frac{\kappa}{R},-\frac{\tau}{R}) +1}{\tau\left(v(\frac{\kappa}{R}, -\frac{\tau}{R})+1\right)^2}\mathbf{I}_p\right]c_n\right\vert  > \varepsilon \right)\mathbb{P}\left(N_n=m\right)\nonumber
        \\
        \leq& \sup_{m\in I_n} \mathbb{P}\left(\frac{1}{\Vert c_n\Vert^2} \left\vert c_n^\top \left[\tau \mathbf{R}\left(\frac{m}{n}\bfSigmahat_{m,n},\tau\right)^2- \frac{v(\frac{\kappa}{R}, -\frac{\tau}{R})-\frac{1}{R}\tau v^\prime(\frac{\kappa}{R},-\frac{\tau}{R}) +1}{\tau\left(v(\frac{\kappa}{R}, -\frac{\tau}{R})+1\right)^2}\mathbf{I}_p\right]c_n\right\vert  > \varepsilon \right),\label{det3}
    \end{align}
    where we used that
    \begin{align*}
        \hat{\mathbf{\Gamma}}_n \mid \{N_n=m\} \sim \bfSigmahat_{m,n},
    \end{align*}
    because $X_{1,n}, \ldots, X_{n,n}$ are i.i.d.. Consequently, \eqref{eq_asymp_variance_second_quad_form_identity_cov} follows, if the right hand side of \eqref{det3} converges to $0$. Suppose that this is not the case. Then, there exists a $\delta > 0$ and a subsequence $n_k$ such that
    \begin{align*}
        \sup_{m\in I_{n_k}} \mathbb{P}\left(\frac{1}{\Vert c_{n_k}\Vert^2} \left\vert c_{n_k}^\top \left[\tau \mathbf{R}\left(\frac{m}{n_k}\bfSigmahat_{m,n_k},\tau\right)^2- \frac{v(\frac{\kappa}{R}, -\frac{\tau}{R})-\frac{1}{R}\tau v^\prime(\frac{\kappa}{R},-\frac{\tau}{R}) +1}{\tau\left(v(\frac{\kappa}{R}, -\frac{\tau}{R})+1\right)^2}\mathbf{I}_p\right]c_{n_k}\right\vert  > \varepsilon \right)>\delta.
    \end{align*}
    This implies that $m_k \in I_{n_k}$ exists such that
    \begin{align}\label{eq_contradiction_subsequence_second_identity_cov}
        \mathbb{P}\left(\frac{1}{\Vert c_{n_k}\Vert^2} \left\vert c_{n_k}^\top \left[\tau \mathbf{R}\left(\frac{m_k}{n_k}\bfSigmahat_{m_k,n_k},\tau\right)^2- \frac{v(\frac{\kappa}{R}, -\frac{\tau}{R})-\frac{1}{R}\tau v^\prime(\frac{\kappa}{R},-\frac{\tau}{R}) +1}{\tau\left(v(\frac{\kappa}{R}, -\frac{\tau}{R})+1\right)^2}\mathbf{I}_p\right]c_{n_k}\right\vert  > \varepsilon \right)>\delta.
    \end{align}
        Because $m_k \in I_{n_k}$ for all $k\in \mathbb{N}$, we obtain
    \begin{align*}
        \left\vert \frac{m_k}{n_k}-R \right\vert \leq \eta_{n_k}\to 0,
    \end{align*}
    which implies
    \begin{align*}
        \frac{p}{m_k} = \frac{p/n_k}{m_k/n_k} \to \frac{\kappa}{R}.
    \end{align*}
    Consequently, there exists a compact set $K \subset \mathbb{C}\setminus \mathbb{R}_{\geq 0}$ such that $\zeta_k= -\frac{n_k}{m_k}\tau \in K$ for sufficiently large $k$. Since $\Vert c_{n_k}c_{n_k}^\top \Vert_{\tr}/\Vert c_{n_k}\Vert^2=1$, it follows by Lemma \ref{lem_derivative_det_equi_identity_cov} that
    \begin{align*}
        &\frac{1}{\Vert c_{n_k}\Vert^2} \left\vert c_{n_k}^\top \left[\mathbf{R}\left(\bfSigmahat_{m_k,n_k},\frac{n_k}{m_k}\tau\right)^2 - \frac{v(\frac{\kappa}{R}, -\frac{n_k}{m_k}\tau)-\frac{n_k}{m_k}\tau v^\prime(\frac{\kappa}{R},-\frac{n_k}{m_k}\tau) +1}{\left(\frac{n_k}{m_k}\tau\right)^2\left(v(\frac{\kappa}{R}, -\frac{n_k}{m_k}\tau)+1\right)^2}\mathbf{I}_p\right]c_{n_k} \right\vert
        \\
        =&\frac{1}{\Vert c_{n_k}\Vert^2} \frac{1}{\tau }\left(\frac{m_k}{n_k}\right)^2 \left\vert c_{n_k}^\top \left[\tau\mathbf{R}\left(\frac{m_k}{n_k}\bfSigmahat_{m_k,n_k},\tau\right)^2 - \frac{v(\frac{\kappa}{R}, -\frac{n_k}{m_k}\tau)-\frac{n_k}{m_k}\tau v^\prime(\frac{\kappa}{R},-\frac{n_k}{m_k}\tau) +1}{\tau\left(v(\frac{\kappa}{R}, -\frac{n_k}{m_k}\tau)+1\right)^2}\mathbf{I}_p\right]c_{n_k} \right\vert \conp 0.
    \end{align*}
    Since $m_k/n_k \to R$, we obtain
    \begin{align*}
        \frac{1}{\Vert c_{n_k}\Vert^2} \left\vert c_{n_k}^\top \left[\tau\mathbf{R}\left(\frac{m_k}{n_k}\bfSigmahat_{m_k,n_k},\tau\right)^2 - \frac{v(\frac{\kappa}{R}, -\frac{n_k}{m_k}\tau)-\frac{n_k}{m_k}\tau v^\prime(\frac{\kappa}{R},-\frac{n_k}{m_k}\tau) +1}{\tau\left(v(\frac{\kappa}{R}, -\frac{n_k}{m_k}\tau)+1\right)^2}\mathbf{I}_p\right]c_{n_k} \right\vert \conp 0.
    \end{align*}
    Since $v(\kappa, z)$ is holomorphic in $z\in \mathbb{C}\setminus \mathbb{R}_{\geq 0}$, it follows that
    \begin{align*}
        \frac{1}{\Vert c_{n_k}\Vert^2}\left\vert c_{n_k}^\top \left[\left(\frac{v(\frac{\kappa}{R}, -\frac{n_k}{m_k}\tau)-\frac{n_k}{m_k}\tau v^\prime(\frac{\kappa}{R},-\frac{n_k}{m_k}\tau) +1}{\tau\left(v(\frac{\kappa}{R}, -\frac{n_k}{m_k}\tau)+1\right)^2} - \frac{v(\frac{\kappa}{R}, -\frac{\tau}{R})-\frac{1}{R}\tau v^\prime(\frac{\kappa}{R},-\frac{\tau}{R}) +1}{\tau\left(v(\frac{\kappa}{R}, -\frac{\tau}{R})+1\right)^2}\right)\mathbf{I}_p\right]c_{n_k} \right\vert \conp 0.
    \end{align*}
    Consequently,
    \begin{align*}
        \frac{1}{\Vert c_{n_k}\Vert^2} \left\vert c_{n_k}^\top \left[\tau\mathbf{R}\left(\frac{m_k}{n_k}\bfSigmahat_{m_k,n_k},\tau\right)^2 - \frac{v(\frac{\kappa}{R}, -\frac{\tau}{R})-\frac{1}{R}\tau v^\prime(\frac{\kappa}{R},-\frac{\tau}{R}) +1}{\tau\left(v(\frac{\kappa}{R}, -\frac{\tau}{R})+1\right)^2}\mathbf{I}_p\right]c_{n_k} \right\vert \conp 0.
    \end{align*}
    This is a contradiction to \eqref{eq_contradiction_subsequence_second_identity_cov}, and therefore the right hand side of \eqref{det3} converges to $0$.

    \smallskip

    Finally, we show that the limit
    \begin{align*}
        \frac{\frac{1}{R}v^\prime\left(\frac{\kappa}{R},-\frac{\tau}{R}\right)}{\left(v\left(\frac{\kappa}{R}, -\frac{\tau}{R}\right)+1\right)^2} >0.
    \end{align*}
    Let $F^{\kappa,H}$ denote the limiting distribution of the empirical spectral distribution of $\bfSigmahat_n$ while $\underline{F}^{\kappa, H}$ denote the limiting distribution of the empirical spectral distribution of the companion sample covariance matrix $\underline{\bfSigmahat}_n$. In particular,
    \begin{align}\label{eq_companion_lsd}
        \underline{F}^{\kappa, H} = \begin{cases}
            \kappa F^{\kappa, H} + (1-\kappa) \delta_{\{0\}} & \kappa \leq 1
            \\
            \kappa F^{\kappa,H} - (\kappa -1) \delta_{\{ 0\}} & \kappa > 1
        \end{cases}.
    \end{align}
    Then, it follows that
    \begin{align*}
        v(\kappa, z) = \int \frac{1}{t-z} \, d\underline{F}^{\kappa,H}(t), \quad z\notin \supp(\underline{F}^{\kappa, H}),
    \end{align*}
    which leads to
    \begin{align*}
        v^\prime(\kappa, z) = \int \frac{1}{(t-z)^2} \, d\underline{F}^{\kappa,H}(t), \quad z\notin \supp(\underline{F}^{\kappa, H}).
    \end{align*}
    Consequently, we deduce that $\sigma^2 >0$.

    \subsection{Proof of Corollary \ref{cor_test_performance}}
    Note, $\{N_n=0\} \subset \{ \hat{\sigma}_n^2 =0\}$. However, Theorem \ref{thm_consistent_variance_estimator_identity_cov} implies that
    \begin{align*}
        \mathbb{P}(\hat{\sigma}_n^2=0)\to 0.
    \end{align*}
    Consequently, it suffices to restrict our attention to the event $\{\hat{\sigma}_n^2>0\} \subset \{N_n>0\}$ throughout the following proof.
    \newline
    \textit{Proof of Corollary \ref{cor_test_performance} (b).}
    By Theorem \ref{thm_main_CLT_identity_cov}, we have
    \begin{align}\label{eq_breve_tn_asymp_normal}
        \Breve{T}_n &= \frac{1}{\sigma}\frac{\sqrt{r_n n}}{\Vert c_n\Vert} c_n^\top \left(\hat{\beta}-(1-a_n(N_n))\beta_0\right) \cond \mathcal{N}(0,1).
    \end{align}
    Note that
    \begin{align}\label{eq_Tn_rep_under_alt}
        T_n = \frac{\sigma}{\hat{\sigma}_n}\left(\Breve{T}_n +D_n\right),
    \end{align}
    where 
    \begin{align*}
        D_n = \frac{1}{\sigma}\frac{\sqrt{r_n n}}{\Vert c_n\Vert}\left(1-a_n(N_n)\right)\left(c_n^\top \beta_0 -\theta_n\right).
    \end{align*}
    Since $\Breve{T}_n = O_{\mathbb{P}}(1)$ and $\sigma/\hat{\sigma}_n \conp 1$ by Theorem \ref{thm_consistent_variance_estimator_identity_cov}, it is sufficient to show
    \begin{align*}
        \vert D_n \vert \conp \infty, \quad n \to \infty
    \end{align*}
    to conclude that
    \begin{align}\label{eq_power_goal_identity_cov}
        \mathbb{P}\left(\vert T_n \vert > q_{1-\frac{\alpha}{2}}\mid \mathsf{H}_{1,n}\right) \to 1, \quad n \to \infty.
    \end{align}
    Recall that $\{\hat{\sigma}_n^2 >0\} \subset \{N_n >0\}$. Hence, we have
    \begin{align*}
        \vert D_n\vert  = \frac{\sqrt{r_n n}}{\Vert c_n\Vert}\frac{v\left(\frac{p}{N_n},-\frac{n}{N_n}\tau\right)}{v\left(\frac{p}{N_n},-\frac{n}{N_n}\tau\right)+1}\vert c_n^\top \beta_0-\theta_n\vert.
    \end{align*}
    Because $p/N_n \conp \kappa/R$ and $n/N_n \conp 1/R$, Lemma \ref{lem_sequence_stieltjes_transform_arguments_continuous} implies that
    \begin{align*}
        \frac{v\left(\frac{p}{N_n},-\frac{n}{N_n}\tau\right)}{v\left(\frac{p}{N_n},-\frac{n}{N_n}\tau\right)+1} \conp \frac{v\left(\frac{\kappa}{R},-\frac{\tau}{R}\right)}{v\left(\frac{\kappa}{R},-\frac{\tau}{R}\right)+1}.
    \end{align*}
    Recall that $v\left(\frac{\kappa}{R},z\right)>0$ for all $z\in \mathbb{R}_{<0}$. Consequently, we obtain
    \begin{align*}
        \frac{v\left(\frac{\kappa}{R},-\frac{\tau}{R}\right)}{v\left(\frac{\kappa}{R},-\frac{\tau}{R}\right)+1} > 0.
    \end{align*}
    Since we assume that $\frac{\sqrt{n}}{\Vert c_n\Vert}\vert c_n^\top \beta_0 - \theta_n\vert \to \infty$, we deduce that \eqref{eq_power_goal_identity_cov} holds.

    \medskip
    
    \textit{Proof of Corollary \ref{cor_test_performance} (c).}
    Because of the representation \eqref{eq_Tn_rep_under_alt}, Theorem \ref{thm_consistent_variance_estimator_identity_cov} and \eqref{eq_breve_tn_asymp_normal}, it suffices to show that
    \begin{align}\label{eq_Dn_converges_mu}
        D_n \conp \mu.
    \end{align}
    With analogous arguments as in the previous proof, we obtain
    \begin{align*}
       \left(1-a_n(N_n)\right) \conp \frac{v(\kappa/R,-\tau/R)}{1+v(\kappa/R,-\tau/R)}.
    \end{align*}
    Combining this result with the assumptions \begin{align*}
        \frac{\sqrt{n}}{\|c_n\|}(c_n^\top\beta_0-\theta_n) \longrightarrow d
    \end{align*}
    and $r_n \to R$, we deduce \eqref{eq_Dn_converges_mu}.

    \subsection{Proof of Theorem \ref{thm_CLT_arbitrary_cov}}
    We start with some preparations. First, we show that
    \begin{align}\label{eq_variance_zero_negligble}
    \mathbb{P}\left(\{\sigma_n^2= 0\}\right)\to 0.
    \end{align}
    \textit{Proof of \eqref{eq_variance_zero_negligble}.}
    Recall that $\mathbf{B}_n$ is positive semidefinite. Therefore, the matrix $\mathbf{R}(\mathbf{B}_n, \tau)$ commutes with $\mathbf{B}_n^{1/2}$, which implies that
    \begin{align*}
    \mathbb{P}\left(\{\sigma_n^2= 0\}\right) = \mathbb{P}\left(\frac{\sigma_{\varepsilon}^2}{\Vert c_n\Vert^2} \left\Vert \mathbf{B}_n^{1/2}\mathbf{R}(\mathbf{B}_n, \tau)c_n\right\Vert^2=0\right) = \mathbb{P}\left(\frac{1}{\Vert c_n\Vert^2} c_n^\top \mathbf{B}_n c_n = 0\right).
\end{align*}
By the independence from assumption \ref{ass_random_vectors_MP} and \ref{ass_mutual_independence}, we have
\begin{align*}
    \mathbb{E}\left[\frac{1}{\Vert c_n\Vert^2}c_n^\top \mathbf{B}_n c_n\right] = \mathbb{E}\left[\frac{1}{\Vert c_n\Vert^2}\frac{1}{n}\sum_{i=1}^{n} \delta_i c_n^\top X_iX_i^\top c_n \right]=\frac{r_n }{\Vert c_n\Vert^2} c_n^\top \bfSigma_n c_n
\end{align*}
and
\begin{align*}
    \Var\left(\frac{1}{\Vert c_n\Vert^2}c_n^\top \mathbf{B}_n c_n\right) &= \Var\left(\frac{1}{\Vert c_n\Vert^2}\frac{1}{n}\sum_{i=1}^{n} \delta_i c_n^\top X_iX_i^\top c_n \right) = \frac{1}{\Vert c_n\Vert^4}\frac{1}{n^2} \sum_{i=1}^{n}\Var\left(\delta_i c_n^\top X_i X_i^\top c_n\right)
    \\
    &\leq \frac{1}{\Vert c_n\Vert^4} \frac{1}{n} \mathbb{E} \left[(c_n^\top X_1 X_1^\top c_n)^2\right] = \frac{1}{\Vert c_n\Vert^4}\frac{1}{n}\mathbb{E}\left[(Z_1^\top \bfSigma_n^{1/2}c_n)^4\right].
\end{align*}
The uniform eighth-moment bound in \ref{ass_random_vectors_MP} and the spectral norm bound in \ref{ass_spectral_norm_cov_matrix} imply that
\begin{align*}
    \mathbb{E}\left[(Z_1^\top \bfSigma_n^{1/2}c_n)^4\right] &= \mathbb{E}\left[\sum_{j=1}^{p} Z_1(j)^4 (e_j^\top\bfSigma_n^{1/2}c_n)^4 +3\sum_{i\neq j}^{p} Z_1(i)^2Z_1(j)^2(e_i^\top\bfSigma_n^{1/2}c_n)^2(e_j^\top\bfSigma_n^{1/2}c_n)^2 \right]
    \\
    &\leq \sqrt{M_8}\sum_{j=1}^{p} (e_j^\top \bfSigma_n^{1/2}c_n)^4 +3\sum_{i\neq j}^{p}(e_i^\top \bfSigma_n^{1/2}c_n)^2(e_j^\top \bfSigma_n^{1/2}c_n)^2
    \\
    &\leq \max\{3,\sqrt{M_8}\} \Vert \bfSigma_{n}^{1/2}c_n\Vert^4 \lesssim \Vert c_n\Vert ^4.
\end{align*}
Hence, 
\begin{align*}
    \Var\left(\frac{1}{\Vert c_n\Vert^2}c_n^\top \mathbf{B}_n c_n\right)\lesssim \frac{C}{n}.
\end{align*}
If \eqref{eq_ass_c_norm} and \eqref{eq_ass_c_quadratic_form} are satisfied, we obtain
\begin{align*}
    \frac{1}{\Vert c_n\Vert^2}c_n^\top \bfSigma_n c_n \geq \frac{\inf_{n\in \mathbb{N}}c_n^\top \bfSigma_n c_n}{\sup_{n\in \mathbb{N}}\Vert c_n\Vert^2} > 0.
\end{align*}
If \eqref{eq_eigenvalues_bounded_from_below} is satisfied, then we have
\begin{align*}
    \frac{1}{\Vert c_n\Vert^2}c_n^\top \bfSigma_n c_n \geq m > 0.
\end{align*}
Therefore, there exists a constant $\gamma > 0$ such that
\begin{align}\label{eq_TL_C1}
    \mathbb{E}\left[\frac{1}{\Vert c_n\Vert^2}c_n^\top \mathbf{B}_n c_n\right] \geq \frac{R}{2}\gamma
\end{align}
for sufficiently large $n$, where we used the fact that $r_n \to R$.
Consequently, using Chebyshev's inequality yields
\begin{align*}
    \mathbb{P}\left(\{\sigma_n^2 =0\}\right)&= \mathbb{P}\left(\frac{1}{\Vert c_n\Vert^2}c_n^\top \mathbf{B}_n c_n= 0\right) 
    \\
    &\leq \mathbb{P}\left(\frac{1}{\Vert c_n\Vert^2}\left\vert  c_n^\top \mathbf{B}_n c_n - \mathbb{E}\left[c_n^\top \mathbf{B}_n c_n\right]\right\vert\geq \frac{R\gamma}{4} \right)
    \\
    &\lesssim \frac{1}{n} \to 0.
\end{align*}
As a consequence of \eqref{eq_variance_zero_negligble}, it suffices to restrict our attention to the event $\{\sigma_n^2 > 0\}$ throughout the following proof.
Therefore, we have 
\begin{align*}
    T_n = \sum_{i=1}^{n}\frac{1}{\sigma_n}\frac{\sqrt{r_n}}{\Vert c_n\Vert\sqrt{n}} c_n^\top \mathbf{R}\left(\mathbf{B}_n , \tau\right)X_i \delta_i \varepsilon_i.
\end{align*}
It suffices to show the conditional CLT
\begin{align*}
    T_n \mid \mathbf{X}_n, \Delta_n \cond \mathcal{N}(0,R).
\end{align*}
In particular, we show
\begin{align}\label{eq_conditional_Ljapunov_variance_arbitrary_cov}
    s_n^2 = \sum_{i=1}^{n}\Var \left(\frac{1}{\sigma_n}\frac{\sqrt{r_n}}{\Vert c_n\Vert\sqrt{n}} c_n^\top \mathbf{R}\left(\mathbf{B}_n , \tau\right)X_i \delta_i \varepsilon_i\mid \mathbf{X}_n, \Delta_n\right) \conp R
\end{align}
and
\begin{align}\label{eq_conditional_Ljapunov_sum_arbitrary_cov}
    \sum_{i=1}^{n}\mathbb{E}\left[\left(\frac{1}{\sigma_n}\frac{\sqrt{r_n}}{\Vert c_n\Vert\sqrt{n}} c_n^\top \mathbf{R}\left(\mathbf{B}_n , \tau\right)X_i \delta_i \varepsilon_i\right)^4\mid \mathbf{X}_n, \Delta_n\right] \conp 0.
\end{align}
\textit{Proof of \eqref{eq_conditional_Ljapunov_variance_arbitrary_cov}.}
Recall that $\varepsilon_1, \ldots, \varepsilon_n$ are independent of $\mathbf{X}_n, \Delta_n$ by assumption \ref{ass_mutual_independence} with $\mathbb{E}[\varepsilon_1^2] = \sigma_{\varepsilon}^2$. Hence, we obtain
\begin{align*}
    s_n^2&= \frac{1}{\sigma_n^2} \frac{r_n\sigma_{\varepsilon}^2}{\Vert c_n\Vert^2 n} \sum_{i=1}^{n} (c_n^\top \mathbf{R}(\mathbf{B}_n, \tau)X_i \delta_i)^2
    \\
    &= \frac{1}{\sigma_n^2}\frac{r_n\sigma_{\varepsilon}^2}{\Vert c_n\Vert^2 n} \sum_{i=1}^{n} c_n^\top \mathbf{R}(\mathbf{B}_n, \tau)X_i \delta_i X_i^\top \mathbf{R}(\mathbf{B}_n, \tau)c_n = \frac{1}{\sigma_n^2}\frac{r_n \sigma_{\varepsilon}^2}{\Vert c_n\Vert^2 }c^\top \mathbf{R}(\mathbf{B}_n, \tau) \mathbf{B}_n \mathbf{R}(\mathbf{B}_n,\tau)c_n
    \\
    &= \frac{1}{\sigma_n^2}\frac{r_n \sigma_{\varepsilon}^2}{\Vert c_n\Vert^2 } c_n^\top \left[\mathbf{R}(\mathbf{B}_n, \tau) - \tau \mathbf{R}(\mathbf{B}_n, \tau)^2\right]c_n = r_n \to R.
\end{align*}

\smallskip

\textit{Proof of \eqref{eq_conditional_Ljapunov_sum_arbitrary_cov}.}
By the assumptions \ref{ass_random_errors} and \ref{ass_mutual_independence} it follows that
\begin{align}
    \sum_{i=1}^{n}\mathbb{E}\left[\left(\frac{1}{\sigma_n}\frac{\sqrt{r_n}}{\Vert c_n\Vert\sqrt{n}} c_n^\top \mathbf{R}\left(\mathbf{B}_n , \tau\right)X_i \delta_i \varepsilon_i\right)^4\mid \mathbf{X}_n, \Delta_n\right] &=\frac{1}{\sigma_n^4}\frac{r_n^2}{\Vert c_n\Vert^4 n^2}\sum_{i=1}^{n} \left(c_n^\top \mathbf{R}(\mathbf{B}_n, \tau) X_i \delta_i\right)^4 \mathbb{E}\left[\varepsilon_i^4\right]\nonumber
    \\
    &\lesssim \frac{1}{\sigma_n^4}\frac{r_n^2}{\Vert c_n\Vert^4 n^2}\sum_{i=1}^{n} \left(c_n^\top \mathbf{R}(\mathbf{B}_n, \tau) X_i\right)^4.\label{eq_TL_B1}
\end{align}
We will show that
\begin{align}\label{eq_rosenthal_arbitrary_cov}
    \frac{1}{\Vert c_n\Vert^4 n^2}\sum_{i=1}^{n} \left(c_n^\top \mathbf{R}(\mathbf{B}_n,\tau)X_i\right)^4 = O_{\mathbb{P}}(n^{-1})
\end{align}
and
\begin{align}\label{eq_inv_asymp_var_is_bounded}
    \frac{1}{\sigma_n^4} = O_{\mathbb{P}}(1).
\end{align}
Applying these results to \eqref{eq_TL_B1} yields \eqref{eq_conditional_Ljapunov_sum_arbitrary_cov}.

\smallskip

\textit{Proof of \eqref{eq_rosenthal_arbitrary_cov}.}
With similar arguments as in \eqref{eq_Sherman_Woodbury_id_cov}, we obtain
\begin{align*}
    \mathbb{E}\left[\left\vert c_n^\top \mathbf{R}(\mathbf{B}_n, \tau)X_i\right\vert^4\right] &\leq \mathbb{E}\left[\left\vert c_n^\top \mathbf{R}(\mathbf{B}_n^{(i)},\tau)X_i\right\vert^4\right], \quad i=1,\ldots,n.
\end{align*}
Consequently, for $i=1,\ldots,n$ we have
\begin{align*}
    \mathbb{E}\left[\left\vert c_n^\top \mathbf{R}(\mathbf{B}_n, \tau)X_i\right\vert^4\right] &\leq \mathbb{E}\left[\left\vert c_n^\top \mathbf{R}(\mathbf{B}_n^{(i)},\tau)X_i\right\vert^4\right]=\mathbb{E}\left[\left\vert c_n^\top \mathbf{R}(\mathbf{B}_n^{(i)}, \tau)\bfSigma_n^{1/2}Z_i\right\vert^4\right] 
    \\
    &=\mathbb{E}\left[\mathbb{E}\left[\left\vert\sum_{j=1}^{p} c_n^\top \mathbf{R}(\mathbf{B}_n^{(i)},\tau)\bfSigma_n^{1/2}e_jZ_i(j)\right\vert^4\mid \mathbf{B}_n^{(i)}\right]\right]
    \\
    &\lesssim \mathbb{E}\left[\mathbb{E}\left[ \sum_{j=1}^{p} \left\vert c_n^\top \mathbf{R}(\mathbf{B}_n^{(i)},\tau)\bfSigma_n^{1/2}e_jZ_i(j)\right\vert^4\mid \mathbf{B}_n^{(i)}\right]\right]
    \\
    & ~~~+ \mathbb{E}\left[\mathbb{E}\left[\sum_{j=1}^{p} \vert c_n^\top \mathbf{R}(\mathbf{B}_n^{(i)},\tau)\bfSigma_n^{1/2}e_j Z_i(j)\vert^2\mid \mathbf{B}_n^{(i)}\right]^2\right]
    \\
    &\lesssim \mathbb{E}\left[\sum_{j=1}^{p} \left\vert c_n^\top \mathbf{R}(\mathbf{B}_n^{(i)},\tau)\bfSigma_n^{1/2}e_j \right\vert^4\right] +\mathbb{E}\left[\left(\sum_{j=1}^{p} \left\vert c_n^\top \mathbf{R}(\mathbf{B}_n^{(i)},\tau)\bfSigma_n^{1/2}e_j \right\vert^2\right)^2\right]
    \\
    &\leq \mathbb{E}\left[\left\Vert \bfSigma_n^{1/2}\mathbf{R}(\mathbf{B}_n^{(i)}, \tau) c_n
    \right\Vert^4\right] +\mathbb{E}\left[\left\Vert \bfSigma_n^{1/2}\mathbf{R}(\mathbf{B}_n^{(i)}, \tau) c_n
    \right\Vert^4\right]
    \\
    &\lesssim \Vert c_n\Vert^4,
\end{align*}
where the last inequality follows from assumption \ref{ass_spectral_norm_cov_matrix} together with the bound \(\|\mathbf R(\mathbf B_n^{(i)},\tau)\|_2\le\tau^{-1}\).
Hence, we obtain
\begin{align}\label{eq_TL_B2}
    \mathbb{E}\left[\frac{1}{\Vert c_n\Vert^4 n^2} \sum_{i=1}^{n} \left(c_n^\top \mathbf{R}(\mathbf{B}_n, \tau)X_i\right)^4\right] \lesssim \mathbb{E}\left[\frac{1}{n}\right]=O(n^{-1}),
\end{align}
which shows \eqref{eq_rosenthal_arbitrary_cov}.

\smallskip

\textit{Proof of \eqref{eq_inv_asymp_var_is_bounded}.}
Note, by the definition of $\sigma_n^2$, we have the following representation
\begin{align*}
    \sigma_n^2 = \frac{\sigma_{\varepsilon}^2}{\Vert c_n\Vert^2}c_n^\top \mathbf{B}_n \mathbf{R}(\mathbf{B}_n, \tau)^2 c_n.
\end{align*}
Since $\mathbf{B}_n$ is nonnegative definite, we have
\begin{align*}
    \frac{1}{(\Vert \mathbf{B}_n\Vert_2 + \tau)^2}\mathbf{B}_n\preceq \mathbf{B}_n \mathbf{R}\left(\mathbf{B}_n, \tau\right)^2,
\end{align*}
which implies that
\begin{align*}
    \frac{\sigma_{\varepsilon}^2}{\Vert c_n\Vert^2}\frac{1}{(\Vert \mathbf{B}_n\Vert_2 + \tau)^2}c_n^\top \mathbf{B}_n c \leq \sigma_n^2. 
\end{align*}
Recall that $\{\sigma_n^2 > 0\} = \{ \frac{1}{\Vert c_n\Vert^2}c_n^\top \mathbf{B}_n c_n > 0\}$. Hence, it follows that
\begin{align*}
    \frac{1}{\sigma_n^2} \leq \frac{\Vert c_n\Vert^2 \left(\Vert \mathbf{B}_n \Vert_2 + \tau\right)^2}{\sigma_{\varepsilon}^2 c_n^\top \mathbf{B}_nc_n}.
\end{align*}
Now, define $E_n=\{\frac{1}{\Vert c_n\Vert^2} c_n^\top \mathbf{B}_n c_n\geq \frac{R\gamma}{4}\}\subset \{\sigma_n^2 > 0\}$, where $\gamma$ is defined in \eqref{eq_TL_C1}.
\begin{align}\label{eq_TL_B4}
    \lim_{K\to \infty}\limsup_{n\to \infty}\mathbb{P}\left(\left\{\frac{1}{\sigma_n^2}>K\right\}\right) \leq \lim_{K\to \infty}\limsup_{n\to \infty} \left(\mathbb{P}(E_n^c) + \mathbb{P}\left(\frac{(\Vert \mathbf{B}_n\Vert_2 + \tau)^2}{\sigma_{\varepsilon}^2}\frac{4}{R \gamma}>K\right)\right).
\end{align}
We will show that the right hand side of \eqref{eq_TL_B4} converges to $0$.
Because of \eqref{eq_TL_C1}, it follows that
\begin{align*}
    \mathbb{P}\left(E_n^c\right)=\mathbb{P}\left(\frac{1}{\Vert c_n\Vert^2}c^\top \mathbf{B}_n c_n < \frac{R \gamma}{4}\right) &\leq \mathbb{P}\left(\frac{1}{\Vert c_n\Vert^2}\left\vert c_n^\top \mathbf{B}_n c_n - \mathbb{E}\left[c_n^\top \mathbf{B}_n c_n\right]\right\vert \geq \frac{R \gamma}{4}\right) 
    \\
    &\lesssim \Var \left(\frac{1}{\Vert c_n\Vert^2} c_n^\top \mathbf{B}_n c_n\right)\lesssim \frac{1}{n}\to 0.
\end{align*}
Since Lemma \ref{lem_sample_cov_op_norm_non_identical_distributed} yields $\Vert\mathbf{B}_n\Vert_2 \leq \Vert \bfSigmahat_n\Vert_2=O_{\mathbb{P}}(1)$, it follows that
\begin{align*}
    \lim_{K\to \infty}\limsup_{n\to \infty}\mathbb{P}\left(\frac{(\Vert \mathbf{B}_n\Vert_2 + \tau)^2}{\sigma_{\varepsilon}^2}\frac{4}{R \gamma}>K\right) = 0.
\end{align*}
Therefore, the right hand side of \eqref{eq_TL_B4} converges to $0$, and completes the proof of \eqref{eq_inv_asymp_var_is_bounded}.
\subsection{Proof of Corollary \ref{cor_variance_estimator_arbitrary_cov}}
As a consequence of \eqref{eq_variance_zero_negligble}, it suffices to restrict our attention to the event $\{\sigma_n^2>0\}$ throughout the following proof. Therefore, the definitions of $\sigma_n^2$ and $\tilde{\sigma}_n^2$ yield
\begin{align*}
    \frac{\tilde{\sigma}_n^2}{\sigma_n^2} = \frac{\hat{\sigma}_{n,1}^2}{\sigma_{\varepsilon}^2}.
\end{align*}
Since \eqref{eq_consistent_noise_variance} remains valid for general covariance matrices $\bfSigma_n$ satisfying our assumptions and $\sigma_{\varepsilon}^2 >0$, it follows that
\begin{align*}
    \frac{\hat{\sigma}_{n,1}^2}{\sigma_{\varepsilon}^2} \conp 1,
\end{align*}
which proves the assertion.
\section{Auxiliary results}

In this section, we state several auxiliary results, which are required for the proofs of our main statements. 

\begin{lemma}\label{lem_demoninator_var_est_away_from_zero}
    Suppose that the assumptions \ref{ass_sample_to_dimension_ratio} - \ref{ass_mutual_independence} are satisfied. Let $N_n= \sum_{i=1}^{n} \delta_{i,n}$. Then, we have
    \begin{align*}
        \liminf_{n\to \infty} \frac{N_n-\tr(\mathbf{M}_n)}{n} >0
    \end{align*}
    almost surely.
\end{lemma}

\begin{proof}
    Note, by Hoeffding's inequality and Borel-Cantelli it follows that
    \begin{align}\label{eq_subsampled_ratio_difference_as_convergence}
        \frac{N_n}{n} - r_n \conas 0.
    \end{align}
    This leads to
    \begin{align}\label{eq_subsampled_ratio_as_convergence}
        \frac{N_n}{n}\conas R,
    \end{align}
    since $r_n \to R$. Consequently, we restrict our attention to the event $\{N_n > 0\}$ throughout the remaining proof.
    Applying Jensen's inequality with $f(x)=\tau/(x+\tau)$ in \eqref{eq_TL1}, we obtain
    \begin{align}\label{eq_TL1a}
        \frac{N_n - \tr(\mathbf{M}_n)}{n} = \frac{1}{n}\sum_{i=1}^{N_n} \frac{\tau}{\lambda_i +\tau} \geq \frac{N_n}{n} \frac{\tau}{\tau + \frac{1}{N_n} \sum_{i=1}^{N_n}\lambda_i} = \frac{\tau \left(\frac{N_n}{n}\right)^2}{\tau \frac{N_n}{n}+ \frac{1}{n}\tr(\mathbf{B}_n)}.
    \end{align}
    
     Therefore, it remains to investigate $1/n \tr(\mathbf{B}_n)$. Since
    \begin{align*}
        \frac{1}{n} \tr\left(\mathbf{B}_n\right) = \frac{1}{n^2}\sum_{i=1}^{n} \delta_{i,n} \Vert X_{i,n}\Vert^2 \leq \frac{1}{n^2}\sum_{i=1}^{n} \Vert X_{i,n}\Vert^2,
    \end{align*}
    we will show that
    \begin{align}\label{eq_TL2}
        \limsup_{n\to \infty}\frac{1}{n^2} \sum_{i=1}^{n} \Vert X_{i,n} \Vert^2 \leq \limsup_{n\to \infty} \frac{\tr(\bfSigma_n)}{n}
    \end{align}
    almost surely holds. This leads to
    \begin{align*}
        \limsup_{n\to \infty} \frac{1}{n}\tr(\mathbf{B}_n) \leq \limsup_{n\to \infty} \frac{\tr(\bfSigma_n)}{n} \leq \kappa \sup_{n\in \mathbb{N}} \Vert \bfSigma_n\Vert_2
    \end{align*}
    almost surely. Finally, using \eqref{eq_subsampled_ratio_as_convergence} and \eqref{eq_TL2} in \eqref{eq_TL1a} yields that
    \begin{align*}
        \liminf_{n\to \infty} \frac{N_n-\tr(\mathbf{M}_n)}{n} \geq \frac{\tau R^2}{\tau R + \kappa \sup_{n\in \mathbb{N}}\Vert \bfSigma_n\Vert_2} > 0
    \end{align*}
    almost surely holds. 
    
    \smallskip

    \textit{Proof of \eqref{eq_TL2}.}
    By the representation $X_{i,n} = \bfSigma_n^{1/2} Z_{i,n}$, we obtain
    \begin{align*}
        \mathbb{E}\left[\Vert X_{i,n}\Vert^2\right] = \mathbb{E}\left[Z_{i,n}^\top \bfSigma_n Z_{i,n}\right] = \tr(\bfSigma_n).
    \end{align*}
    Hence,
    \begin{align*}
        \Var\left(\Vert X_{i,n} \Vert^2\right) &= \mathbb{E}\left[\left(Z_{i,n}^\top \bfSigma_n Z_{i,n} - \tr(\bfSigma_n)\right)^2\right] 
        \\
        &= \mathbb{E}\left[\left(\sum_{j=1}^{p} \bfSigma_n(j,j)(Z_{i,n}(j)^2-1) + \sum_{j\neq k}^{p} \bfSigma_n(j,k) Z_{i,n}(j)Z_{i,n}(k)\right)^2\right].
    \end{align*}
    By the independence of the entries of $Z_{i,n}$ and the fact that the entries are centered, we obtain
    \begin{align*}
        \mathbb{E}\left[\left(\sum_{j=1}^{p}\bfSigma_n(j,j)(Z_{i,n}(j)^2-1)\right)^2\right] = \sum_{j=1}^{p} \bfSigma_n(j,j)^2 \left(\mathbb{E}[Z_{i,n}(j)^4]-1\right)
    \end{align*}
    and
    \begin{align*}
        \mathbb{E}\left[\left(\sum_{j\neq k}^{p} \bfSigma_n(j,k)Z_{i,n}(j)Z_{i,n}(k)\right)^2\right] =\sum_{j\neq k}^{p} \left(\bfSigma_n(j,k)^2 + \bfSigma_n(j,k)\bfSigma_n(k,j)\right) = 2\sum_{j\neq k}^{p} \bfSigma_n(j,k)^2 
    \end{align*}
    and
    \begin{align*}
        \mathbb{E}\left[\sum_{j=1}^{p}\sum_{k\neq \ell}^{p}\bfSigma_n(j,j)\bfSigma_n(k,\ell)(Z_{i,n}(j)^2-1)Z_{i,n}(k)Z_{i,n}(\ell)\right] = 0.
    \end{align*}
    Therefore, for $i=1,\ldots, n$, we have
    \begin{align*}
        \Var(\Vert X_{i,n}\Vert^2) &= \sum_{j=1}^{p} \bfSigma_n(j,j)^2 \left(\mathbb{E}[Z_{i,n}(j)^4]-1\right) + 2 \sum_{j\neq k}^{p} \bfSigma_n(j,k)^2
        \\
        & \lesssim \Vert \bfSigma_n\Vert_F^2 \leq p \Vert \bfSigma_n\Vert_2^2 = O(p).
    \end{align*}
    By Markov's inequality we obtain
    \begin{align*}
        \mathbb{P}\left(\left\vert \frac{1}{n^2}\sum_{i=1}^{n} \Vert X_{i,n} \Vert^2 - \frac{\tr(\bfSigma_n)}{n}\right\vert>\varepsilon\right) \leq\frac{1}{\varepsilon^2 n^4} \Var\left(\sum_{i=1}^{n} \Vert X_{i,n}\Vert^2\right) = \frac{1}{\varepsilon^2 n^4} \sum_{i=1}^{n} \Var\left( \Vert X_{i,n}\Vert^2\right) \lesssim \frac{p}{\varepsilon^2 n^3},
    \end{align*}
    where we used that $X_{1,n}, \ldots X_{n,n}$ are independent. Consequently, Borel-Cantelli implies that
    \begin{align*}
        \frac{1}{n^2} \sum_{i=1}^{n} \Vert X_{i,n}\Vert^2 - \frac{\tr(\bfSigma_n)}{n}\conas 0,
    \end{align*}
    which shows that \eqref{eq_TL2} almost surely holds.
\end{proof}

\begin{lemma}\label{lem_sample_cov_op_norm_non_identical_distributed}
Suppose that assumptions \ref{ass_sample_to_dimension_ratio} - \ref{ass_spectral_norm_cov_matrix} are satisfied. Then,
\begin{align*}
    \left\Vert \bfSigmahat_n \right\Vert_2 = O_{\mathbb{P}}(1),
\end{align*}
where $\bfSigmahat_n= \frac{1}{n}\sum_{i=1}^{n} X_{i,n}X_{i,n}^\top$.
\end{lemma}
Lemma \ref{lem_sample_cov_op_norm_non_identical_distributed} follows from Theorem $2$ in \cite{Latala2005} together with Markov's inequality and the uniform spectral norm bound in assumption \ref{ass_spectral_norm_cov_matrix}.
\newline
For the following auxiliary results, we consider that $\bfSigma_n = \mathbf{I}_p$. Hence, the functions $v(\kappa, z)$ and $v^\prime(\kappa,z)$ are given by \eqref{eq_def_v_function_identity_cov} and \eqref{eq_def_v_derivative_function_identity_cov}, respectively. The first statement  corresponds to the isotropic case of Theorem $2.1$ of \cite{StarkDoernemann2026}.

\begin{lemma}\label{lem_det_equi_DS_identity_cov}
    Suppose that the assumptions \ref{ass_sample_to_dimension_ratio} - \ref{ass_ESD_conv} are satisfied. Let $\left\{\Theta_n\right\}_{n \geq 1}$ be a sequence of nonrandom $p\times p$ matrices. Then, for any compact set $K \subset \mathbb{C}\setminus \mathbb{R}_{\geq 0}$,
    \begin{align*}
        \sup_{\zeta\in K}\left |\tr\left( \Theta_{n}\left\{\mathbf{R}(\bfSigmahat_n, -\zeta) - \frac{1}{-\zeta v\left(\kappa_n,\zeta\right)-\zeta}\mathbf{I}_p\right\}\right)\right | =O_{\mathbb{P}}\left(\frac{\Vert \Theta_n\Vert_F}{\sqrt{n}} \right),
    \end{align*}
    where $\bfSigmahat_n=\frac{1}{n}\sum\limits_{i=1}^{n} X_{i,n}X_{i,n}^\top$.
\end{lemma}

\begin{lemma}\label{lem_det_equi_identity_cov}
    Suppose that the assumptions \ref{ass_sample_to_dimension_ratio} - \ref{ass_ESD_conv} are satisfied. Let $\{\Theta_n\}_{n\geq 1}$ be a sequence of nonrandom $p\times p$ matrices satisfying $\sup_{n\geq 1}\Vert \Theta_n\Vert_{\tr}= \sup_{n\geq 1}\tr((\Theta_n^\top \Theta_n)^{1/2}) < \infty$. Then, for any compact set $K \subset \mathbb{C}\setminus \mathbb{R}_{\geq 0}$,
    \begin{align*}
        \sup_{\zeta \in K}\left | \tr \left(\Theta_n \left\{\mathbf{R}(\bfSigmahat_n, -\zeta)-\frac{1}{-\zeta v(\kappa,\zeta)-\zeta}\mathbf{I}_p\right\}\right)\right |  = o_{\mathbb{P}}(1).
    \end{align*}
\end{lemma}
Lemma \ref{lem_det_equi_identity_cov} follows from Lemma \ref{lem_det_equi_DS_identity_cov}, since $\sup_{n \geq 1} \Vert \Theta_n\Vert_{\tr} < \infty$ and $\kappa_n \to \kappa$, together with the local uniform convergence of $v(\kappa_n, \cdot)$ to $v(\kappa, \cdot)$.
\begin{lemma}\label{lem_derivative_det_equi_identity_cov}
     Suppose that the assumptions \ref{ass_sample_to_dimension_ratio} - \ref{ass_ESD_conv} are satisfied. Let $\{\Theta_n\}_{n\geq 1}$ be a sequence of nonrandom $p\times p$ matrices satisfying $\sup_{n\geq 1}\Vert \Theta_n\Vert_{\tr}= \sup_{n\geq 1}\tr((\Theta_n^\top \Theta_n)^{1/2}) < \infty$. Then, for any compact set $K \subset \mathbb{C}\setminus \mathbb{R}_{\geq 0}$,
    \begin{align*}
        &\sup_{\zeta \in K} \left\vert \frac{d}{d\zeta}\tr\left[ \Theta_n \left\{\mathbf{R}(\bfSigmahat_n, -\zeta)-\frac{1}{-\zeta v(\kappa,\zeta)-\zeta}\mathbf{I}_p\right\}\right]\right\vert
        \\
        =& \sup_{\zeta \in K} \left\vert \tr\left[\Theta_n \left\{\mathbf{R}(\bfSigmahat_n, -\zeta)^2- \frac{v(\kappa, \zeta)+\zeta v^\prime(\kappa, \zeta)+1}{\left(-\zeta v(\kappa, \zeta) - \zeta\right)^2}\mathbf{I}_p \right\}\right]\right\vert =o_{\mathbb{P}}(1).
    \end{align*}
\end{lemma}
Lemma \ref{lem_derivative_det_equi_identity_cov} follows from Corollary $2.1$ in \cite{StarkDoernemann2026} with $q=1$, using the local uniform convergence of $v(\kappa_n,\cdot)$ and its first derivative with respect to $\zeta$.

\begin{lemma}\label{lem_sequence_stieltjes_transform_arguments_continuous}
    Let $(\kappa_n, z_n)_{n\geq 1}$ be a sequence in $(0,\infty)\times \supp(\underline{F}^{\kappa, \delta_{\{1\}}})^c \cap \mathbb{R}_{<0}$ with limit $(\kappa, z) \in (0,\infty)\times \supp(\underline{F}^{\kappa, \delta_{\{1\}}})^c \cap \mathbb{R}_{<0}$. Then, 
    \begin{align*}
        v(\kappa_n,z_n) \to v(\kappa,z).
    \end{align*}
\end{lemma}

\begin{proof}
    Since 
    \begin{align*}
        v_n(\kappa_n,z_n) = v(\kappa_n,z_n) = \frac{-(1-\kappa_n+z_n)-\sqrt{(1-\kappa_n+z_n)^2-4z_n}}{2z_n},
    \end{align*}
    the claim follows due to continuity.
\end{proof}

{\bf Funding } The work of HD and TL has been  supported by the Deutsche Forschungsgemeinschaft (DFG); Research unit 5381 \textit{Mathematical Statistics in the Information Age} (460867398).
The work of ND is partially supported by  the Aarhus University Research
Foundation (AUFF),  project numbers 47221 and 47388, and by a RECRUIT Sabbatical
Grant from the Novo Nordisk Foundation, project number 52629. ND gratefully acknowledges the hospitality of the Department of
Statistics at Harvard University during the preparation of this
manuscript.

\setlength{\bibsep}{1pt}
\begin{small}
\bibliography{references}
\end{small}
\end{document}